\documentclass[12pt]{article}
\usepackage{amssymb}
\usepackage{blindtext}
\usepackage{hyperref}
\usepackage{setspace}
\usepackage{float}
\hypersetup{
	colorlinks=true,
	linkcolor=blue,
}
\usepackage[automark]{scrlayer-scrpage}
\usepackage[utf8]{inputenc}  

\usepackage{tocloft}
\usepackage{blindtext}
\usepackage{subcaption}
\usepackage[shortcuts]{extdash}

\usepackage{amsthm}
\usepackage{enumitem}
\usepackage[ruled,vlined]{algorithm2e}
\usepackage{amsmath}           
\usepackage{amsfonts}          
\usepackage{amssymb}
\usepackage{mathalfa}           
\usepackage{graphicx}          
\usepackage[T1]{fontenc}       
\usepackage{ae}                

\usepackage{lastpage}          
\usepackage[margin=10pt,font=small,labelfont=bf]{caption} 
\usepackage{enumitem}
\usepackage{stmaryrd}
\usepackage[arrow, matrix, curve]{xy}
\usepackage{listings}
\usepackage{color} 
\definecolor{mygreen}{RGB}{28,172,0} 
\definecolor{mylilas}{RGB}{170,55,241}

\usepackage{tikz}
\usetikzlibrary{knots,calc,shapes.geometric}
\usetikzlibrary{positioning}

\newcommand{\RR}{\mathbb{R}}

\newcommand{\st}{\quad\textrm{s.t.}\quad}

\newcommand*{\norm}[1]{\left\Vert#1\right\Vert}

\newtheorem{prop}{Proposition}[section]
\newtheorem{thm}[prop]{Theorem}
\newtheorem{defn}[prop]{Definition}
\newtheorem{lem}[prop]{Lemma}

\newtheorem{algo}[prop]{Algorithm}

\theoremstyle{definition}

\newtheorem{example}[prop]{Example}

\DeclareMathAlphabet{\pazocal}{OMS}{zplm}{m}{n}

\makeatletter
\newcommand\nocaption{%
    \renewcommand\p@subfigure{}
    \renewcommand\thesubfigure{\thefigure\alph{subfigure})}
}
\makeatother

\makeatother

\usepackage{pdfpages}
\usepackage{layout}
\usepackage{comment}

\begin{document}
\author{Christian Kanzow\thanks{University of Würzburg, Institute of Mathematics, Campus Hubland Nord,
         Emil-Fischer-Str.\ 30, 97074 Würzburg, Germany; christian.kanzow@uni-wuerzburg.de} \and 
         Felix Wei{\ss}\thanks{University of Würzburg, Institute of Mathematics, Campus Hubland Nord,
         Emil-Fischer-Str.\ 30, 97074 Würzburg, Germany; felix.weiss@uni-wuerzburg.de}}
\title{A Modified Exact Penalty Approach for General Constrained $ \ell_0 $-Sparse Optimization Problems}
\date{September 2, 2025}
\maketitle

{
\noindent 
\small\textbf{Abstract.} We consider a general class of constrained
optimization problems with an additional $ \ell_0 $-sparsity term in the 
objective function. Based on a recent reformulation of this
difficult $ \ell_0 $-term, we consider a nonsmooth penalty approach
which differs from the authors' previous work by the fact that 
it can be directly applied to problems which do not necessarily
contain nonnegativity constraints. This avoids a splitting of
free variables into their positive and negative parts, reduces
the dimension and fully exploits the one-to-one correspondence
between local and global minima of the given $ \ell_0 $-sparse
optimization problem and its reformulation. The penalty approach is shown
to be exact in terms of minima and stationary points. Since the
penalty function is (mildly) nonsmooth, we also present practical
techniques for the solution of the subproblems arising within the
penalty formulation. Finally, the results of an extensive numerical
testing are provided. \\

\noindent
\small\textbf{Keywords.}
Sparse optimization; global minima; local minima;
stationary points; exact penalty function\\

\noindent
\small\textbf{Mathematical Science Classification.}
65K05, 90C26, 90C30, 90C46
\par\addvspace{\baselineskip}
}

\section{Introduction}\label{Sec:Intro}

In this paper, we consider the sparse optimization problem of the form
\begin{equation}
	\min_x f(x) + \rho \norm{x}_0, \quad g(x) \in X, \tag{SPO} \label{SPO}
\end{equation} 
where $f \colon \RR^n \to \RR $ and $g \colon \RR^n \to \RR^m$ are 
smooth functions, $ X \subseteq 
\RR^n $ is a nonempty and closed set, $ \rho > 0 $ a given scalar, and
\begin{equation*}
	\norm{x}_0 := \textrm{number of nonzero components } x_i 
	\textrm{ of } x.
\end{equation*}
Solving the \eqref{SPO} is known to be a difficult task, due to the
combinatorial nature of the optimization problem induced by the 
nonconvex and discontinuous $\ell_0$-norm. It was shown in 
\cite{9030937}, that \eqref{SPO} belongs to the class of 
NP-hard problems, when paired with a quadratic function $f$.

The sparse optimization problem is closely related to the cardinality constrained optimization problem, where the $ \ell_0 $-term in the 
objective function is replaced by a constraint of the form $\norm{x}_0 \le \kappa$. In the survey article \cite{Tillmann2021}, it was shown that any solution of \eqref{SPO} solves the corresponding cardinality constrained problem with $\kappa = \norm{x^*}$. However, in general, it is not true 
that a solution $x^*$ of the cardinality constrained problems solves \eqref{SPO} for some suitable $\rho$. Hence, we cannot substitute the program \eqref{SPO} by a suitable cardinality constrained optimization
problem.

Following \cite{le2015dc}, solution strategies for $\ell_0$-norm 
regularized problems can be divided into the following categories: 
(a) convex approximations, (b) nonconvex approximations, and 
(c) nonconvex exact reformulations. The first approach typically
replaces the $ \ell_0 $-term by the convex $ \ell_1 $-norm and is
particularly popular for LASSO-type problems, cf.\ \cite{Tibshirani2018}.
On the other hand, in many applications, this technique does not
yield the desired sparsity, and sometimes it is not even of any help
in reducing the sparsity like for portfolio optimization problems, cf.\
\cite{KanzowWeiß2023}. The techniques based on (b) try to approximate the 
$ \ell_0 $-term by a suitable (smoother) function. Popular
realizations of this idea use the $ \ell_p $-quasi norm for 
$ p \in (0,1) $, the SCAD (= smoothly clipped absolute deviation) \cite{FanLi2001}, MCP (= minimax concave penalty) \cite{Zhang2010}, 
PiE (= piecewise exponential) \cite{Nguyen2015}, or the transformed $ \ell_1 $ function \cite{ZhangYin2018}, see also the list of functions
in \cite{Cui-Pang-2021}. Note that all these approximations of the
$ \ell_0 $-quasi norm are at least continuous, most of them even
Lipschitz continuous.

Regarding category (c), there exist different exact reformulations
of problem \eqref{SPO}. Assuming that there are finite lower and upper bounds
as part of the standand constraints $ g(x) \in X $, it is straightforward 
to obtain a mixed-integer formulation of the original problem, see \cite{Bienstock1996, Tillmann2021}. Hence, for a convex quadratic $f$
and linear constraints, mixed-integer QP solvers can 
exploit such a reformulation in order to find a global minimum (given 
enough CPU time). Furthermore, there exist reformulations obtained by DC-approaches (DC = difference of convex functions) as seen, for instance, in \cite{le2015dc}, where DC-functions are used to approximate the $\ell_0$-norm, whereas an exact reformulation of the $\ell_0$-norm (though mainly in the context of cardinality-constrained 
problems) was introduced in \cite{Gotoh2017}. Finally, the paper
\cite{Feng2018} presents a formulation with the $ \ell_0 $-term b
being replaces by complementarity constraints. 

The formulation from \cite{Feng2018} is the basis of our approach 
and was already expanded in \cite{KanzowSchwartzWeiss2022}, 
where, in particular, we provide a complete equivalence of local and global minima between the original problem \eqref{SPO} and the corresponding
optimization problem with complementarity constraints. Note that this
is in contrast to a related approach introduced in \cite{BKS-2016}
for cardinality-constrained programs where an equivalence
between local minima does not hold. The recent paper 
\cite{KanzowWeiß2023} generalizes the technique from 
\cite{Feng2018,KanzowSchwartzWeiss2022} by introducing a whole class
of reformulations of \eqref{SPO}, all of which rely on the aforementioned
idea of replacing the $ \ell_0 $-term by suitable complementarity constraints. This class of reformulations was then used in order
to motivate an exact penalty approach for the solution 
of \ref{SPO}.

However, the exact penalty technique from \cite{KanzowWeiß2023} assumes
that the general constraints $ g(x) \in X $ contain nonnegativity
constraints on the variables $ x $. Otherwise, free variables have
to be splitted into the positive and negative parts, which then
increases the number of variables and destroys the shown one-to-one
correspondence between (local) minima. Furthermore, this splitting
might violate suitable constraint qualifications which are required
to verify convergence results for the underlying exact penalty approach.

In this paper, we further elaborate on the general class of reformulations 
from \cite{KanzowWeiß2023} and, in particular, present a penalty approach
which can be applied to the general problem \eqref{SPO} without
the explicit assumption that the variables $ x $ are nonnegative.
In particular, we show that our new penalty approach is exact in terms
of (local/global) minima and, most importantly, also with respect
to stationary points. In contrast to the exact penalty approach 
from \cite{KanzowWeiß2023}, the objective function of the resulting
penalty subproblems is (mildly) nonsmooth. Nevertheless, we will show
that these subproblems can still be solved efficiently by suitable
methods.

The paper is structured as follows: In Section~\ref{Sec:Background}, we
state some background material from variational analysis, recall 
the details for the reformulation of problem \eqref{SPO} from the 
paper \cite{KanzowWeiß2023}, and consider a so-called tightened
nonlinear program associated to a local minimum of
\eqref{SPO}. The new exact penalty approach and related exactness results 
are given in Section~\ref{section:penalty}.
In particular, this includes exactness with respect to stationary points
which, from a practical point of view, are particularly relevant. 
Section~\ref{section:solstrat} presents two solution strategies which 
aim to find stationary points of the penalized subproblems, one is 
based on a projected gradient method, the other one on a proximal 
gradient-type approach. We then present some numerical results in 
Section~\ref{section:numerics} based on a variety of different
classes of applications, namely portfolio optimization, dictionary learning, and adversarial attacks on neural networks.´We close with some final
remarks in Section~\ref{Sec:FinalRemarks}.
 
Notation: In the various parts of this paper, we address via 
$$
   I_0(x):= \{ i \, | \, x_i = 0\}
$$ 
the set of indices for which $x$ vanishes. Furthermore, we write $x \circ y$ for the Hadamard product of $x$ and $y$, i.e.\ the componentwise multiplication of the two vectors. We abbreviate the canonical unit vector by $e_i \in \RR^{n}$, indicating that the single $1$ is in the $i$-th position, and additionally write $e := (1,1,...,1)^T \in \RR^{n}$. Since we will introduce sign constraints to our variables, we also denote with $\RR^n_{\ge 0}$ the cone of vectors with only nonnegative entries in $\RR^n$.

\section{Mathematical Background}\label{Sec:Background}

\subsection{Variational Analysis}

Here, we recall some results and definitions from the field of 
variational analysis that will be used in our subsequent theory.

To this end, let $X \subseteq \mathbb{R}^m$ be a nonempty and closed set. The \emph{(Bouligand) tangent cone} at a point $x \in X$ to the set $X$ is given by
$$
   T_X(x):= \{d \in\mathbb{R}^m \ : \ \exists d^k \to d, \exists 
   t_k\searrow 0 \ : x+t_kd^k \in X\}.
$$
The polar
$$
   N^F_X(x) := T_X(x)^\circ := \{v \in \mathbb{R}^m \ : \ v^Td \le 0, \ \forall d\in T_X(x)\}
$$
is the \emph{Fr\'echet normal cone} of $ x \in X $.
The \emph{limiting normal cone} by Mordukhovich is obtained 
from the Fr\'echet normal cone by passing to the outer limit, i.e.
$$
   N^{\lim}_X(x) := \limsup_{\bar{x} \to x} N^F_X(\bar{x}) := \{v \in \mathbb{R}^m \ : \ \exists x^k \to_X x, \ \exists v^k \to v \ : v^k \in N_X^F(x^k) \ \forall k\in \mathbb{N} \}.
$$
On the other hand, the standard \emph{normal cone (from convex analysis)} 
to a closed and convex set $X \subset \RR^n$ is defined by
$$
   N_X(x):=\{v \in \RR^n : v^T(y - x) \le 0, \forall y \in X\}
$$
for any $ x \in X $. Note that all these cones coincide for convex
sets, hence, we have $N_X(x) = N_X^F(x) = N_X^{\lim}(x)$ for all $x \in X$
with $ X $ closed, and convex. For the sake of completeness,
all these cones are defined to be empty at any point $ x \not\in X $.

We also need a notion of differentiability for a nonsmooth function,
for which different choices are possible. Here, we take the 
limiting subdifferential by Mordukhovich: Given a lower semicontinuous (lsc) function $f:\mathbb{R}^n \to \mathbb{R}$, we define
$$
   \partial f(x) :=\{ s \in \mathbb{R}^n \ : \ (s,-1) \in N_{epi(f)}^{\lim}((x,f(x)))\},
$$
where $ \text{epi}(f) $ denotes the epigraph of $f$. We call
$ \partial f(x) $ the \emph{limiting subdifferential} by Mordukhovich,
and each element $ s \in \partial f(x) $ a (limiting) subgradient.
This limiting subdifferential has some useful calculus rules. Some
of the basic ones, that will be used in our subsequent analysis,
are summarized in the following result.

\begin{lem}[Calculus for the limiting subdifferential]\label{lem:calcsub}
	We have the following calculus rules:
	\begin{enumerate}[label = (\roman*)]
		\item Let $f(x) = \sum_{i = 1}^n f_i(x_i),$ where $f_i : \RR \to \RR$ is lsc. Then
		$$\partial f(x) = \partial f_1(x_1) \times \partial f_2(x_2) \times \dots \times \partial f_n(x_n).$$
		\item Let $f : \RR^n \to \RR$ be locally Lipschitz and let $\varphi : \RR^n \to \RR$ be lsc. Then
		$$\partial(f + \varphi)(x)  \subset \partial f(x) + \partial\varphi(x).$$
		\item Let $f:\RR^n \to \RR$ be $C^1$ and $\varphi : \RR^n \to \RR$ be lsc. Then
		$$\partial \left(f + \varphi\right)(x) = \partial f(x) + \partial\varphi(x) = \nabla f(x) +\partial\varphi(x).$$
		\item Let $f : \RR^n \to \RR$ be $C^1$ and $\varphi : \RR^m \to \RR^n$ be locally Lipschitz continuous. At a point $x$, set $y = \varphi(x)$. It holds
		$$\partial (f \circ \varphi)(x) = \partial \left(\nabla f(y)^T \varphi\right)(x),$$
		where the subdifferential is taken with respect to the function $\nabla f(y)^T \varphi : \RR^n \to \RR$.
		\item Let $\varphi : \RR^n \to \RR$ be locally Lipschitz continuous and $g : \RR^m \to \RR^n$. Then
		$$\partial (\varphi \circ g)(x) \subset g'(x)^T\partial\varphi(x).$$
	\end{enumerate}
\end{lem}

\begin{proof}
Statement (i) (on separable functions) follows from \cite[Prop 10.5]{Rockafellar1998}. The sum rule (ii) (with an 
inclusion) can be found in
\cite[Thm. 3.36]{Mordukhovich2006} and results from the fact that,
if one of the two functions is locally Lipschitz, then the corresponding
singular subdifferential from that reference contains the zero vector only.
The sum rule from (iii) (with equality) is taken from
\cite[Prop. 1.107]{Mordukhovich2006}. The subsequent chain rules
(iv) and (v) can be found in \cite[Thm. 1.110]{Mordukhovich2006}
and \cite[Thm. 10.6]{Rockafellar1998}, respectively.
\end{proof}

As a subdifferential of particular interest for our purposes, we have
the following one.

\begin{lem}\label{Lem:Subdiff}
Consider the function $ \varphi (x) := \rho \norm{x}_0 $ for some
$ \rho > 0 $. Then 
\begin{equation*}
	\partial \varphi (x) =
	\big\{ s \in \RR^n \, \big| \, s_i = 0 \textrm{ for all } i 
	\text{ with } x_i \neq 0 \big\} 
\end{equation*}
holds for all $ x \in \RR^n $.
\end{lem}

\begin{proof}
See for instance \cite{Le2012,Durea2014}.
\end{proof}

We next introduce the dist function to an arbitrary set $X\subset \RR^n$ 
as the minimal distance of $x$ to this set, measured by an arbitrary norm,
i.e.
$$
   \text{dist}_X(x) := \inf \{\lVert y - x\rVert \ : \ y \in X\}.
$$
In case $X$ is closed, the infimum is, of course, attained (albeit
not necessarily unique), and we can replace the infimum by a minimum in 
this case. The dist function is famously known as an exact penalty function.
To this end, we first clarifiy what we consider as exactness from here on.

\begin{defn}
Let a problem $\min_{x \in C} f(x)$ be given and let $\phi(x)$ denote some merit function, which fulfills $\phi(x) = 0$ if $x \in C$ and $\phi(x) > 0$ otherwise. We call $\text{P}_{\alpha} (x) := f(x) + \alpha \phi(x)$ an
\emph{exact penalty function} if, for any (local) minimizer $x^*$ of $f$ over $C$, there is a finite value $\alpha^*>0$ such that, for all $\alpha \ge \alpha^*$, the point $x^*$ is also a (local) minimizer of $P_\alpha$.
\end{defn}

Note that, in the previous definition and the subsequent statements within
this subsection, the function $ f $ is not necessarily the objective
function from our given optimization problem \eqref{SPO}.

The following central result goes back to Clarke~\cite{Clarke1990} and 
clearly illustrates why the dist function plays such an important role 
for exact penalty results. The statement is phrased in a way that includes the idea of not penalizing the entire feasible set, but rather some constraints with the help of the dist function.

\begin{thm}\label{thm:distexact}
Let the problem $\min_{x\in C} f(x)$ be given, where $C \subset S$ and $f$ is Lipschitz of rank $L$ on $S$. If $x^*$ is a minimizer of $f$ over $C$, then $x^*$ minimizes the function $P_\alpha(x) := f(x) + \alpha \text{dist}_C(x)$ for all $\alpha \ge L$ over $S$. In addition, if $C$ is closed, then, for $\alpha > L$, any other point minimizing $P_\alpha$ over $S$ also lies in $C$ (and therefore minimizes $f$ over $C$).
\end{thm}

This theorem can be extended to local minima and a locally Lipschitz continuous function $f$ simply by restricting $C \cap B_{\varepsilon}(x^*)$, where $B_{\varepsilon}(x^*)$ is a closed ball of radius $\varepsilon$ around $x^*$, such that the local minimizer $x^*$ solves $\min_{C \cap B_\varepsilon(x^*)} f(x)$. In this case $x^*$ is truly a local minimizer of $P_\alpha$ since one can prove the existence of a neighborhood $U$ of $x^*$ such that
$$
   \text{dist}_{C \cap B_\varepsilon(x^*)}(x) = \text{dist}_C(x), \quad \forall x \in U,$$
cf.\ \cite{YE20121642}.
Conversely, for $\alpha > L$ any additional minimizer $\overline{x}$ of $P_\alpha$ over $S \cap U$ fulfills exactly the property $\text{dist}_{C\cap B_\varepsilon(x^*)}(\overline{x}) = 0 \ \Longleftrightarrow \ \overline{x} \in C\cap B_\varepsilon(x^*)$.

Within the context of our sparse optimization problem \eqref{SPO}
the feasible set is given by the preimage $g^{-1}(X)$. Consequently,
we would have to apply the distance function result to the set
$ g^{-1}(X)$. The computationally more favorable choice, however, would be
to switch to the dist function with respect to $X$ at a point $g(x)$.
This leads to an exact merit function precisely under the metric subregularity condition.

\begin{defn}
We say that \emph{metric subregularity} of rank $\kappa > 0$ is fulfilled for $g(x) \in X$ at a feasible point $x^*$, if there is a neighborhood $U$ of $x^*$ such that
$$
	\text{dist}_{g^{-1}(X)} (x) \le \kappa \, \text{dist}_X(g(x)) \quad \forall x \in U. 
$$
\end{defn}

As an immediate consequence, we obtain that, under metric subregularity,
at least locally and for $ f $ being locally Lipschitz, the penalty function
$$
   \hat P(x) = f(x) + \alpha_* \kappa \text{dist}_X(g(x))
$$
is exact. Consequently, given 
a local minimizer $x^*$, one necessarily has $0 \in \partial \hat P(x^*)$
by Fermat's rule for the limiting subdifferential.
Application of Lemma~\ref{lem:calcsub}$(ii)$ then gives \begin{equation} 0 \in \partial f(x^*) + \alpha^*\kappa \partial \text{dist}_X(g(x^*)).\label{eqn:statdist}\end{equation}
The subdifferential of the dist function is well-known.

\begin{lem}\label{lem:subdiffdist}
Let $X \subset \RR^n$ be nonempty and closed, and let $P_X(x) := \{y \in X : \lVert x -y\rVert = \text{dist}_X(x)\}$. Then the subdifferential of the dist function ís given by
$$\partial \text{dist}_X(x) = \begin{cases}
	N^{\lim}_X(x) \cap B_1(0), \quad &x\in X\\
	\frac{x-P_X(x)}{\text{dist}_X(x)}, \quad &x\notin X.
\end{cases}$$
\end{lem}
\begin{proof}
	See \cite[Ex. 8.53]{Rockafellar1998}.
\end{proof}

With the necessary analysis in place, we now formulate the notion of (M-)stationarity as found in literature, cf.\ \cite{Liang2021}

\begin{defn}
Consider the problem 
$$
   P =\min\ f(x) \st g(x)\in X ,
$$
where $f:\mathbb{R}^n \to \mathbb{R}$ is locally Lipschitz, $g:\mathbb{R}^n \to \mathbb{R}^m$ a $C^1$ function and $X\subset \mathbb{R}^m$ a nonempty and closed set. We call $x^*$ a \emph{stationary point} of $P$ if $g(x^*) \in X$ and there exists $\lambda \in N_X^{\lim}(g(x^*))$ such that
$$
   0 \in \partial f(x^*) + g'(x^*)^T \lambda.
$$
\end{defn}

The result of Clarke together with the sum rule from 
Lemma~\ref{lem:calcsub} $(ii)$ implies that a local minizer $x^*$ of $P$ 
satisfies
$$
   0 \in \partial f(x^*) + N_{g^{-1}(X)}^{\lim}(x^*),
$$
which, in general, is an impractical criterion to check. Under metric subregularity, however, we arrive at the necessary optimality condition \eqref{eqn:statdist}, where Lemma \ref{lem:subdiffdist} together with Lemma \ref{lem:calcsub} $(v)$ infers
$$
   0 \in \partial f(x^*) + g'(x^*)^T\lambda, \quad \lambda \in N^{\lim}_X(g(x^*))
$$
at a local minimizer $x^*$. Metric subregurality is therefore not only a criterion for the existence of an exact penalty function, but can also be considered a constraint qualification for a more general class of constraints.

\subsection{Reformulation of Sparse Optimization Problems}

Here, we recall the class of reformulations of problem \eqref{SPO}
from our paper \cite{KanzowWeiß2023}. To this end, $p^{\rho} \colon \RR^n \to \RR $ be a function (usually depending on a parameter $ \rho > 0 $)
given by
\begin{equation}\label{Eq:p-rho}
	p^\rho(y) = \sum_{i=1}^n p^{\rho}_i(y_i)
\end{equation}
with each $ p^{\rho}_i \colon \RR \to \RR $ being such that it satisfies the following conditions:
\begin{itemize}
	\item[(P.1)] $ p^\rho_i$ is convex and attains a unique minimum
	(possibly depending on $ \rho $) at some point $ s_i^\rho > 0 $;
	\item[(P.2)] $p^\rho_i(0) - p^\rho_i(s_i^\rho) = \rho$;
	\item[(P.3)] $p^\rho_i$ is continuously differentiable.
\end{itemize}
Assumption (P.1) simply states that $ p^\rho_i $ is a convex function
which attains its unique minimum in the open interval $ (0, \infty) $. We denote this minimum by $ s_i^\rho $. Furthermore, we write
\begin{equation}\label{Eq:prho-Min}
	m_i^{\rho} := p_i^{\rho} (s_i^{\rho}) \quad \text{and} \quad 
	M^{\rho} := \sum_{i=1}^n m_i^{\rho}
\end{equation}
for the corresponding minimal function values of $ p_i^{\rho} $ and 
$ p^{\rho} $, respectively. Condition (P.2) corresponds to a 
suitable scaling of the function $ p_i^{\rho} $ that can always be guaranteed to hold by a suitable multiplication of $ p_i^{\rho} $.
Finally, condition (P.3) is a smoothness condition. We stress that some
of our results hold under the weaker condition that $ p_i^{\rho} $
is only continuous like, e.g., the subsequent statements of the
equivalence between global and local minima, but the observations
regarding stationary points and the practical solution of the exact
penalty subproblems require the continuous differentiability of 
the mappings $ p_i^{\rho} $.

We also recall some examples of function $ \rho_i^{\rho} $ satisfying
properties (P.1)-(P.3), cf.\ \cite{KanzowWeiß2023} for more details

\begin{example}\label{Ex:pirho}
The following functions $ p_i^{\rho} \colon \RR \to \RR $ satisfy
properties (P.1)-(P.3):
\begin{itemize}
	\item[(a)] $ p_i^{\rho}(y_i) := \rho y_i ( y_i - 2) $;
	\item[(b)] $ p_i^{\rho}(y_i) := \frac{1}{2} (y_i - \sqrt{2 \rho})^2 $;
	\item[(c)] A suitable Huber-type modification of the shifted 
	absolute-value function $ p_i^{\rho} (y_i):= \rho | y_i - 1 | $ can also
	be constructed to satisfy (P.1)-(P.3). (Note that $ p_i^{\rho} $
	itself is nonsmooth and therefore violates (P.3).)
\end{itemize}
\end{example}

In the following, we assume that $ p^{\rho} $ is given by \eqref{Eq:p-rho}
with each term $ p_i^{\rho} $ satisfying conditions (P.1)-(P.3) (though
only continuity is required instead of (P.3) within this subsection). We
then introduce the reformulation
\begin{equation}
   \min_{x,y} f(x) + p^\rho(y), \st g(x)\in X, \ x \circ y = 0, \ y\ge0. \label{SPOref} \tag{SPOref}
\end{equation}
of problem \eqref{SPO}. The following results from \cite{KanzowWeiß2023}
show that \eqref{SPOref} is indeed a reformulation of the given
sparse optimization problem \eqref{SPO}.

\begin{lem}\label{0normprop}
Let $ p^\rho $ be given by \eqref{Eq:p-rho} with each $ p_i^\rho $ satisfying
properties (P.1)--(P.3), and let $ M^{\rho} $ be defined by
\eqref{Eq:prho-Min}. Then the following statements hold:
\begin{itemize}
	\item[(a)] The inequality $ \rho \norm{x}_0 \le p^\rho(y) - M^\rho $
	   holds for any feasible point $ (x,y) $ of \eqref{SPOref}.
	\item[(b)] Equality $ \rho \norm{x}_0 = p^\rho(y) - M^\rho $
	  holds for a feasible point $ (x,y) $ of \eqref{SPOref} if and only if $y_i = s_i^\rho$ for all $i \in I_0(x)$.
	\item[(c)] If $(x^*,y^*)$ is a local minimum of \eqref{SPOref}, we have $y_i^* = s_i^\rho$ for all $i \in I_0(x^*)$.
\end{itemize}
\end{lem}

Part (c) of the previous result is of particular interest since it tells 
us that, at any local minimum $(x^*,y^*)$, the \emph{biactive set} $\{ i : x_i^*=y_i^* = 0\}$ is empty, and that the vector $ y^* $ corresponding
to $ x^* $ is necessarily given by
\begin{equation}\label{Eq:y-star}
	y^*_i = \begin{cases} s_i^\rho, & \textrm{for } i \in I_0(x^*), \\ 
	0, & \text{otherwise}.
    \end{cases}
\end{equation}

We now state the essential relation between \eqref{SPO} and \eqref{SPOref} also found, see, once again, \cite{KanzowWeiß2023} for the corresponding
proof.

\begin{thm}[Equivalence of Global Minima]\label{Thm:EquivGlobalMinima}
A feasible $x^*$ for \eqref{SPO} is a global minimum of \eqref{SPO} if and
only if $(x^*,y^*)$ with $ y^* $ given by \eqref{Eq:y-star} is a global
minimum of \eqref{SPOref}. 
\end{thm}

The corresponding results for local minima also holds (recall that
this is not true, in general, for a similar reformulation
of cardinality-constrained problems, cf.\ \cite{BKS-2016}).

\begin{thm}[Equivalence of Local Minima]\label{Thm:EquivLocalMinima}
A feasible $x^*$ for \eqref{SPO} is a local minimum of \eqref{SPO} if and
only if $(x^*,y^*)$ with $ y^* $ given by \eqref{Eq:y-star} is a local
minimum of \eqref{SPOref}. 
\end{thm}

The previous two results imply that the sparse optimization
problem \eqref{SPO} and its reformulation \eqref{SPOref} are completely
equivalent programs (at least in terms of local and global
minima). This motivates to solve the given nonsmooth (even
discontinuous) minimization problem \eqref{SPO} by solving the
equivalent smooth (continuously differentiable) program \eqref{SPOref}.
Of course, the latter has its difficulties, too, which arises from
the complementarity constraints. Thus, its solution requires 
suitable problem-tailored techniques. On the other hand, we already
stress at this point that the biactive set of the reformulated
program is empty, which helps a lot to prove suitable properties.

\subsection{The Tightened Nonlinear Program}

The objectve function of \eqref{SPO} is, due to the $\ell_0$-norm, 
not locally Lipschitz, which takes away access from standard NLP theory. 
To remedy this effect, we first consider the 
\emph{tightened nonlinear program}
\begin{equation}
	\min f(x) \st g(x)\in X, \ x_i = 0 \ \forall i\in I_0(x^*) \tag{TNLP$(x^*)$} \label{TNLP}
\end{equation}
around a feasible point $x^*$. This tightened problem plays a fundamental
role for the theoretical investigation of sparse optimization problems
(not from a practical point of view since the given $ x^* $ is usually
a minimum and therefore unknown). For example, most (problem-tailored)
constraint qualifications for the sparse nonlinear program
\eqref{SPO} are typically defined by the corresponding (standard)
constraint qualification for \eqref{TNLP}, cf.\ our previous
works \cite{KanzowSchwartzWeiss2022,KanzowWeiß2023} for some examples.
Here, we will exploit this tightened program in a similar context.

First of all observe that $x^*$ is a local minimum of \eqref{SPO} if and only if $x^*$ solves the corresponding \ref{TNLP}. This, in particular,
motivates to define stationary points also via the tightened program.

\begin{defn}
Let $x^*$ be feasible for \eqref{SPO}. We say that $x^*$ is 
\emph{stationary} for \eqref{SPO} if $x^*$ is stationary for \ref{TNLP}.
\end{defn}

If $f$ is differentiable, this stationarity conditoins collapses to 
$$
   0 = \nabla f(x) + \sum_{I_0(x^*)}\gamma_ie_i + g'(x)^T\lambda, \quad \gamma\in\mathbb{R}^{|I_0(x^*)|}, \ \lambda \in N_X^{\lim}(g(x)).
$$
In turn, at $x=x^*$, this is equivalent to
\begin{equation*}
   0 \in \nabla f(x^*) + \partial \rho \lVert x^*\rVert_0 + g'(x^*)^T \lambda, \quad \lambda \in N_X^{\lim}(g(x^*)).
\end{equation*}
This is \emph{precisely} the (M-)stationarity extended to the case of \eqref{SPO}. Any known result from standard NLP theory, like formulating suitable constraint qualifications and also the existence of exact penalty functions, can now be extended to \eqref{SPO} by means of the corresponding \ref{TNLP} at the point $x^*$.

Let us introduce that \emph{filter matrix}
\begin{equation}\label{Eq:Pmatrix}
	P \in \RR^{|I_0(x^*)| \times n} \text{ whose columns are
	given by } e_i^T \ (i \in I_0(x^*))
\end{equation}
at a feasible point $ x^* $. Note that $ P $ depends on the given point
$ x^* $, though this is not made precise by our notation since the 
underlying vector $ x^* $ will always be clear from the context.
Using this matrix $ P $, the feasible set of \eqref{TNLP} can be
rewritten as
\begin{equation*}
	C := \{x \ : \ Px = 0, \ g(x) \in X \} ,
\end{equation*}
which clarifies that $ P $ is a kind of (sparse) filter of $ x^* $.
Further writing 
\begin{equation*}
	S := g^{-1} (X),
\end{equation*}
the feasible set $ C $ gets the representation 
\begin{equation*}
   C = \{ x \ : \ Px = 0 , x \in S \},
\end{equation*}
so that $ C $ is an intersection of a linear constraint and an abstract
(geometric) constraint. Theorem~\ref{thm:distexact} now tells us that the dist function is already exact. However, we want to keep $S$ as a constraint and penalize only the sparsity inducing part $Px = 0$. The question, whether this leads to an exact penalty formulation, is answered by the well kown error bound property from \cite{Gfrerer2013}, specifically formulated for the \eqref{TNLP}.

\begin{defn}
We say that the \emph{error bound property} for \eqref{TNLP} holds at 
$x^* \in C$ if there exist $\varepsilon>0$ and $\mu > 0$ satisfying
$$
   \text{dist}_C(x) \le \mu \lVert Px\rVert
$$
for all $x \in g^{-1}(X) \cap B_\varepsilon(x^*)$, where $ P $
denotes the matrix from \eqref{Eq:Pmatrix}.
\end{defn}

One can easily check that metric subregularity for the specific structure 
of $C$ already implies the error bound property, so that we will be 
content with the usual constraint qualifications formulated for $C$ 
which imply metric subregularity. Consequently, any conditions 
stronger than the error bound property yields that the penalty function
\begin{equation}
	f(x) + \alpha\norm{Px}_1, \label{PenSPOalpha}
\end{equation}
is exact for \eqref{TNLP}. So far, extensive research has been put into constraint qualifications that imply error boundedness or are equivalent to the existence of error bounds. For example, \cite{MiT-2011} shows
that a quasi-normality conditions is sufficient for an error bound.
The very recent report \cite{andreani2025primal} generalizes this
result and shows that a slightly weaker version of quasi-normality,
called directional quasi-normality,
is actually equivalent to an error bound condition. Note, however, that
the results are presented for functional constraints only.

Here, we view the feasibility condition $ g(x) \in X $ as an abstract constraint and only penalize the seemingly difficult constraints. This will
be done based on a generalized quasi-normality conditions which, in
particular, holds under the more standard generalized MFCQ assumption
to be introduced in the following section.

\section{Exact Penalty Approach for the Sparse Optimization Problem} \label{section:penalty}

We recall from Theorems~\ref{Thm:EquivGlobalMinima} and \ref{Thm:EquivLocalMinima} that the given sparse optimization
problem \eqref{SPO} is completely equivalent to the reformulated
program \eqref{SPOref} both in terms of global and local minima.
In this section, we exploit this equivalence and consider an 
exact penalty approach based on the structure of the reformulated
program \eqref{SPOref}. More precisely, we consider the penalty problem
\begin{equation} \tag{Pen($\alpha$)} \label{Penalpha}
    \min_{x,y} f(x) + p^\rho(y) + \alpha |x|^Ty, \st g(x) \in X, \ y \ge 0,
\end{equation}
where
$$
   |x| := (|x_1|, |x_2|, ..., |x_n|), 
$$
i.e.\ only penalize the complementarity constraints $ x \circ y $ 
from the reformulated program \eqref{SPOref}, whereas we keep the
(seemingly simple) remaining constraints. Note that, and in contrast
to the exact penalty approach from our previous paper \cite{KanzowWeiß2023},
we do not necessarily have that the components of the vector $ x $
are nonnegative, and this is the reason why we need to take the 
absolute value of $ x $ within the penalized objective function.
This, of course, makes this objective function nonsmooth.

To derive stationary conditions for the penalized program
\eqref{Penalpha}, we first present a preliminary result which computes the subdifferential of the nonsmooth term $q(x,y) = |x|^Ty$.

\begin{lem}
At any point $(x,y)$, where $y \ge 0$, the subdifferential of $q(x,y) := |x|^T y$ is given by
\[v \in \partial q(x,y) \quad \Longleftrightarrow \quad v = \begin{pmatrix}
    y \circ s \\ |x|
\end{pmatrix}, \quad s_i \in \begin{cases}
    \{1\}, \quad &x_i > 0,\\
    [-1, 1], \quad &x_i = 0,\\
    \{-1\}, \quad &x_i < 0,
\end{cases} \quad \forall i = 1,...,n.\]
\end{lem}

\begin{proof}
First observe that the function $ q $ is continuously differentiable
with respect to $ y $ with the corresponding partial derivative
given by $ | x | $, which then is equal to the unique second block
component of the subdifferential of $ q $. Regarding the subdifferential
with respect to the variable $ x $, we first note that the function
$ q $ is separable and can be rewritten as
\begin{equation*}
	q(x,y) = \sum_{i=1}^n | x_i | y_i .
\end{equation*}
Application of the subdifferential properties from 
Lemma~\ref{lem:calcsub} and noting that the limiting subdifferential
of a convex function coincides with the standard subdifferential
from convex analysis for a convex function, we obtain the desired statement 
\end{proof}

By application of the sum rule from Lemma~\ref{lem:calcsub} $(iii)$, the limiting subdifferential of the objective function from the penalty problem \eqref{Penalpha} simply reads
\[
	\partial \big( f(x) + p^\rho(y) + \alpha |x|^T y \big) = 
	\begin{pmatrix}
    \nabla f(x) + \alpha y \circ \partial(|x|) \\
    \nabla p^\rho(y) + \alpha |x|,
	\end{pmatrix}
\]
and is locally bounded by continuity of $\nabla f$ and $\nabla p^\rho$ and boundedness of $\partial(|x|)$, which implies local Lipschitz continuity. Hence, any local minimizer $(x,y)$ necessarily satisfies
\[ 0 \in \begin{pmatrix}
    \nabla f(x) + \alpha y \circ \partial(|x|) \\
    \nabla p^\rho(y) + \alpha |x|
\end{pmatrix} + N_{g^{-1}(X)}^{\lim}(x) + N_{\ge 0}(y) \]
where $N_{\ge 0} (y) = \{\lambda \ : \ \lambda^T(z - y) \le 0, \ \forall z \ge 0\} $ denotes the convex normal cone to $\RR_{\ge 0}$. From standard analysis, one has
$$
	\gamma \in N_{\ge 0}(y) \quad \Longleftrightarrow \quad \gamma_i \le 0, \ \gamma_iy_i = 0.
$$
With this in mind, simply writing down that stationarity
condition, especially for \eqref{Penalpha}, leads to the following definition.

\begin{defn}
We call $(x^*,y^*)$ \emph{stationary} for \eqref{Penalpha} if $(x^*,y^*)$ 
is feasible and there are multipliers $\lambda \in N_X^{\lim}(g(x^*))$ and  $\gamma_i\ge0 \ \big(i \in I_0(y^*)\big)$ with
\begin{align} \begin{split} 0 &\in \begin{pmatrix}
	\nabla f(x^*) + \alpha y^* \circ \partial(|x^*|) + g'(x^*)^T\lambda \\
	\nabla p^\rho(y^*) + \alpha |x^*| - \sum_{i \in I_0(y^*)} \gamma_i e_i
	\end{pmatrix}.\label{PenStat}\end{split}
\end{align}
\end{defn}

The following result essentially shows that \eqref{PenSPOalpha} is an exact penalty function of the tightend program \eqref{TNLP} if and only if \eqref{Penalpha} is an exact penalty function of reformulated program \eqref{SPOref}.

\begin{thm}
Let $x^*$ be a local minimum of \eqref{SPO} or, equivalently, $(x^*,y^*)$ be a local minimum of \eqref{SPOref} with $ y^* $ being given by
\eqref{Eq:y-star}. Then \eqref{PenSPOalpha} is an exact penalty function of \eqref{TNLP} at $x^*$ if and only if \eqref{Penalpha} is an exact penalty function of \eqref{SPOref} at $(x^*,y^*).$
\end{thm}

\begin{proof}
Recall that Lemma~\ref{0normprop} implies $y_i^* = s_i^\rho$ for 
$i \in I_0(x^*)$ and $y_i^* = 0$ otherwise. In particular, we therefore
have $\{1,...,n\} = I_0(x^*) \cup I_0(y^*)$.

We first assume that \eqref{PenSPOalpha} be exact for \ref{TNLP}. Then we can find a neighborhood $U = U_x \times U_y$ around $(x^*,y^*)$ and constants $\sigma_x, \, \sigma_y > 0$ such that implication
$$
	(x,y) \in U \quad \Rightarrow \quad |x_i| > \sigma_x, \ y_j > \sigma_y \ \forall (i,j)\in I_0(y^*) \times I_0(x^*), 
$$
holds and, additionally, $x^*$ is a local minimizer of \eqref{PenSPOalpha} in $U_x$ for all sufficiently large $\alpha$. Now take an arbitrary 
element $(x,y) \in U$. Choose $\overline{y}$ such that $\overline{y}_i = y_i$ for $i \in I_0(x^*)$ and $\overline{y}_i = 0$ for $i \in I_0(y^*)$. 
It follows that
$$
   \norm{(x,\bar y) - (x^*,y^*)} \le \norm{(x,y) - (x^*,y^*)},
$$ 
so that we also have $(x, \overline{y}) \in U$. Furthermore, 
it follows that
\begin{align}\label{eq:rho-y-estimate}
\begin{split}
	p^\rho(\bar y) & = \sum_{i=1}^n p_i^\rho(\bar y_i) \ =
	\sum_{i \in I_0 (x^*)} p_i^\rho(\bar y_i) +
	\sum_{i \in I_0 (y^*)} p_i^\rho(\bar y_i) 
	\ = \sum_{i \in I_0 (x^*)} p_i^\rho(y_i) +
	\sum_{i \in I_0 (y^*)} p_i^\rho(0) \\
	& \ge \sum_{i \in I_0 (x^*)} p_i^\rho(s_i^{\rho}) +
	\sum_{i \in I_0 (y^*)} p_i^\rho(0) 
	\ = \sum_{i \in I_0 (x^*)} p_i^\rho(y_i^*) +
	\sum_{i \in I_0 (y^*)} p_i^\rho(y_i^*)
	= p^\rho(y^*),
\end{split}
\end{align}
where the inequality and the subsequent equality follow from the definitions of the minima $ s_i^{\rho} $ and the elements $ y_i^* $, respectively.
Moreover, we can measure a deviation of $y$ to $\overline{y}$ with the 
help of the penalty term $p^\rho$. To this end, fix an index $i \in I_0(y^*)$. Then
\begin{equation}\label{eq:p-rho-nonnegative-partly}
	p_i^\rho(y_i) - p_i^\rho(\overline{y}_i) + \alpha|x_i|y_i -\alpha|x_i|\overline{y}_i = \left(\nabla p^\rho_i(\xi_i) + \alpha|x_i|\right)y_i \ge \left(\nabla p^\rho_i(0) + \alpha\sigma_x\right)y_i,
\end{equation}
where the equation follows from the differential mean-value theorem
together with $ \bar{y}_i = 0 $ for $ i \in I_0(y^*) $, and the inequality
results from the choice of $ \sigma_x $ in combination with the fact that 
$\nabla p^\rho_i(\xi_i) \ge \nabla p^\rho_i(0)$ due to the convexity 
of each $p^\rho_i$ since this implies the monotonicity of its derivate. 
We can assure the right-hand side of the previous expression to be nonnegative by taking
$$
   \alpha \ge \frac{-\nabla p_i^\rho(0)}{\sigma_x} \quad \forall i \in I_0(y^*)
$$
(note that the lower bounds are positive numbers since the functions
$ p_i^{\rho} $ are assumed to attain their unique minimum within the
interval $ (0, + \infty ) $). This choice of $ \alpha $ then guarantees
that 
\begin{equation}\label{eq:p-rho-nonngative}
	p_i^\rho(y_i) - p_i^\rho(\overline{y}_i) + \alpha|x_i|y_i -\alpha|x_i|\overline{y}_i \geq 0 \quad \forall i = 1, \ldots, n
\end{equation}
as this follows from \eqref{eq:p-rho-nonnegative-partly} for all 
$ i \in I_0(y^*) $ (recall that we have $ y_i \ge 0 $), whereas
this holds obviously for all $ i \in I_0(x^*) $ since $ \bar y_i = y_i $
for these indices.
Now let $\overline{\alpha}$ be the very outer penalty parameter that correlates to the exactness of \eqref{PenSPOalpha}. Choose
$$
	\alpha^* \ge \max\left\{\frac{\overline{\alpha}}{\sigma_y}, \max_{i\in I_0(y^*)} \frac{-\nabla p_i^\rho(0)}{\sigma_x} \right\}.
$$
For any $\alpha \ge \alpha^*$, we then obtain
$$
	f(x) + p^\rho(y) + \alpha|x|^Ty \ge f(x) + p^\rho(\overline{y}) + \alpha|x|^T\overline{y} \ge f(x) + p^\rho(y^*) + \alpha\sigma_y\lVert Px\rVert_1 \ge f(x^*) + p^\rho(y^*),
$$
where the first inequality exploits \eqref{eq:p-rho-nonngative},
the second one uses \eqref{eq:rho-y-estimate} together with
\begin{equation*}
	| x |^T \bar y = \sum_{i=1}^n | x_i | \bar y_i = \sum_{i \in I(x^*)}
	| x_i | \bar y_i = \sum_{i \in I(x^*)}
	| x_i | y_i \geq \sigma_y \sum_{i \in I(x^*)} | x_i | = 
	\sigma_y \| P x \|_1,
\end{equation*}
and the final estimate takes into account the assumed exactness of
\eqref{PenSPOalpha} at $ x^* $ (note that $ Px^* = 0 $ by definition of 
the matrix $ P $). This verifies the exactness of \eqref{Penalpha}.

Conversely, assume that \eqref{Penalpha} is exact for \eqref{SPOref} at a local minimizer $(x^*,y^*)$ with neighborhood $U$, again in the form $U = U_x \times U_y$. Hence, we have
\begin{equation}\label{eq:pen-a-1}
	f(x) + p^{\rho}(x) + \alpha | x |^T y \geq f(x^*) + p^{\rho} (x^*) +
	\alpha | x^* |^T y^* = f(x^*) + p^{\rho} (x^*)
\end{equation}
for all $ (x,y) \in U_x \times U_y$ satisfying $ g(x) \in X $ and
$ y \geq 0 $, where the equation follows from the fact that 
$ \alpha | x^* |^T y^* $ by definition of $ y^* $. Now, observe that
\begin{equation}\label{eq:first-inequality}
	f(x) + \alpha \lVert Px\rVert_1 = f(x) + \alpha \sum_{i \in I_0(x^*)}
	| x_i | \geq f(x) + 
	\frac{\alpha}{\max_{j \in I(x^*)}y_j^*} \sum_{i \in I_0(x^*)}y_i^*|x_i|.
\end{equation}
Let $\overline{\alpha}$ denote the outer penalty parameter that correlates to the exactness of \eqref{Penalpha}. Set
$$
	\alpha^* = \overline{\alpha} \cdot \max_{j \in I_0(x^*)} y_j^* 
$$
and take $ \alpha \geq \alpha^* $. Choose and arbitrary $ x \in U_x $
satisfying $ g(x) \in X $. It follows that $ (x,y^*) \in U $.
Hence, we can apply \eqref{eq:pen-a-1} to this vector pair and obtain
\begin{equation*}
	f(x) + p^{\rho} (y^*) + \alpha | x |^T y^* \geq 
	f(x^*) p^{\rho} (y^*),
\end{equation*}
which simplifies to 
\begin{equation}\label{eq:simplifies-to}
	f(x) + \alpha | x |^T y^* \geq f(x^*).
\end{equation}
Altogether, we therefore have
\begin{align*}
	f(x) + \alpha \| Px \|_1 & 
	\geq f(x) + 
	\frac{\alpha}{\max_{j \in I(x^*)}y_j^*} \sum_{i \in I_0(x^*)}y_i^*|x_i| \\
	& = f(x) + \frac{\alpha}{\max_{j \in I(x^*)}y_j^*} | x |^T y^* \\
	& \geq f(x) + \frac{\alpha^*}{\max_{j \in I(x^*)}y_j^*} | x |^T y^* \\
	& = f(x) + \overline{\alpha} | x |^T y^* \\
	& \geq f(x^*).
\end{align*}
where the first inequality is taken from \eqref{eq:first-inequality},
the subsequent equation follows from the definition of $ y^* $, the
next estimate exploits that $ \alpha \geq \alpha^* $, afterwards we
use the definition of $ \alpha^* $, and the final inequality results
from the exact penalty property in \eqref{eq:simplifies-to}.
This proves the exactness of (\ref{PenSPOalpha}).
\end{proof}

Exactness of \eqref{Penalpha} is therefore equivalent to the exactness of \eqref{PenSPOalpha}. Consequently, it suffices to consider constraint qualifications for \eqref{TNLP} that imply the desired error bounds. As mentioned before, we will present the following two conditions.

\begin{defn}
Let $x^*$ be feasible for \eqref{SPO}. We say that the
\begin{enumerate}[label = (\roman*)]
    \item \emph{sparse quasi-normality} is satisfied at $x^*$, if there is no nonzero $\lambda = (\lambda^a, \lambda^b)$ such that
    $$P^T\lambda^a + g'(x^*)^T\lambda^b = 0$$ holds and that satisfies the following condition: there is $\{x^k\} \to x^*$, $\{y^k\} \to g(x^*)$ and $\{\lambda^k\} \to \lambda^b$ such that 
    $$\lambda^k \in N_X^F(y^k), \quad \lambda_i^a \neq 0 \ \Rightarrow \ \lambda_i^a x_i^k > 0, \quad \lambda_i^b \neq 0 \ \Rightarrow \ \lambda_i^b(g_i(x^k) -y_i^k) > 0 .$$
    \item \emph{sparse generalized Mangasarian Fromovitz} (SP-GMFCQ) is satisfied at $x^*$, if there is no nonzero $\lambda = (\lambda^a, \lambda^b)$ such that
    $$P^T \lambda^a + g'(x^*)^T \lambda^b= 0, \quad (\lambda^a, \lambda^b) \in \mathbb{R}^{|I_0(x^*)|} \times N_X^{lim}(x^*).$$
\end{enumerate}
\end{defn}

Obviously, quasi-normality implies the corresponding MFCQ condition, from which we can recover the desired error bounds as well as constraint qualifications to obtain stationarity at a local minimizer.

\begin{thm}
Let $x^*$ be feasible for \eqref{SPO} and assume sparse quasi-normality or SP-GMFCQ to hold at $x^*$. Then the metric subregularity for 
$$(Px, g(x)) \in \{0\}^{|I_0(x^*)|} \times X$$
is satisfied at $x^*$. Furthermore, the error bound property for \eqref{TNLP} is satisfied at $x^*$.
\end{thm}

\begin{proof}
Compare again \cite{Benko2021}.
\end{proof}

In the common setting, where one has $X = \mathbb{R}^m_{\le 0}$, the SP-GMFCQ collapses to SP-MFCQ as presented in our previous work.

Exactness results, as mentioned before, are typically formulated in terms 
of local minima. From a practical point of view, however, corresponding
results on stationary points are significantly more important,
though much less investigated. The following result shows that we
also have exactness in terms of stationary points.

\begin{thm}
Let $(x^*,y^*)$ be stationary for \eqref{SPOref}. Then there exists an $\alpha^* > 0$ such that $(x^*,y^*)$ is a stationary point of \eqref{Penalpha} for all $\alpha \ge \alpha^*$.
\end{thm}

\begin{proof}
By stationarity of $(x^*,y^*)$ for \eqref{SPOref}, we have
\begin{align*}
    0 & = \nabla f(x^*) + g'(x^*)^T\lambda^* + \sum_{i \in I_0(x^*)}\gamma_i^x y_i^* e_i, \\
    0 &= \nabla p^\rho (y^*) + \sum_{i \in I_0(y^*)} \left(\gamma_i^y x_i^* - \nu_i^*\right) e_i
\end{align*}
for suitable Lagrange multipliers $\lambda^*, \gamma^x, \gamma^y, \nu^*$, where $\lambda^* \in N_X(g(x^*))$ and $\nu_i^* \ge 0$. Let 
$$
	\alpha \ge \alpha^* := \max\left\{\max_{i \in I_0(x^*)} |\gamma_i^x|, \max_{i \in I_0(y^*)} |\gamma_i^y| \right\}.
$$ 
Then, clearly, we have
$$
	\gamma_i^x y_i^* \in \alpha y_i^* \partial |x_i^*| \quad \forall i \in I_0(x^*)
$$
since $\partial |x_i^*| = [-1,1]$. Furthermore. for each
$i \in I_0(y^*)$, we get
$$
	0 =\nabla p_i^\rho(0) + \gamma_i^yx_i^* - \nu_i^* = \nabla p_i^\rho(0) + \alpha |x_i^*| - (\nu_i^* + \underbrace{\alpha |x_i^*| - \gamma_i^y x_i^*}_{\ge 0}).
$$
Now set $\lambda = \lambda^*,\ \gamma_i = \nu_i^* + \alpha |x_i^*| - \gamma_i^y x_i^* \geq \nu_i^* \geq 0 $ for $i \in I_0(y^*)$ 
and $\gamma_i = 0$ otherwise. The tuple $(x^*,y^*,\lambda, \gamma)$
satisfies \eqref{PenStat} and $(x^*,y^*)$ is stationary for \eqref{Penalpha}.
\end{proof}

The following result contains an exactness statement for the other direction. This result may be viewed as a generalization of a related theorem given in \cite{Ralph2004} in the context of mathematical programs with equilibrium constraints (MPECs), though our assumptions are weaker.

\begin{thm}\label{thm:exact}
Let $(x^*,y^*)$ be a stationary point of \eqref{SPOref} such that SP-GMFCQ holds at $x^*$. Then there exists an $\alpha^* > 0$ and a neighborhood $U$ of $(x^*,y^*)$ such that for all $\alpha \ge \alpha^*$, every stationary point of \eqref{Penalpha} in $U$ is a stationary point of \eqref{SPOref}.
\end{thm}

\begin{proof}
Assume, by contradiction, that there is a sequence $\alpha_k \to \infty$ and a sequence $\{(x^k,y^k)\}$ such that $(x^k,y^k) \to (x^*,y^*)$, where $(x^k,y^k)$ is stationary for \eqref{Penalpha} with $\alpha = \alpha_k$, but not stationary for \eqref{SPOref}. The proof will be carried out in two parts.
\begin{enumerate}[label = (\roman*)]
    \item We show that $y_i^k$ terminates in $0$ for sufficiently large $k$ if $i \in I_0(y^*)$.
    \item We show that, under SP-MFCQ, $x^k_i$ also terminates in $0$ for sufficiently large $k$ if $i \in I_0(x^*)$.  
\end{enumerate}
The strict complementarity between $x^k$ and $y^k$ is then a consequence of strict complementarity of $(x^*,y^*)$. It follows that $(x^k,y^k)$ is feasible for \eqref{SPOref} and, by stationarity with respect to \eqref{Penalpha}, therefore also stationary for \eqref{SPOref}, which then finishes the contradiction.

(i): Let $i \in I_0(y^*)$ and assume $y_i^k \searrow 0$. By stationarity of \eqref{Penalpha} we have
$$0 = \nabla p_i^\rho(y_i^k) + \alpha_k |x_i^k|.$$
Since $(x^*,y^*)$ is stationary for \eqref{SPOref}, we have, by strict complementarity, that $x_i^* \neq 0$ so that $\alpha_k |x_i^k| \to \infty$, whereas $\nabla p_i^\rho(y_i^k)$ remains bounded and the right hand side in the above equation diverges, which is a contradiction. This implies that $y_i^k = 0$ for almost all $k$, if $y_i^* = 0$. 

(ii): We also claim that $x_i^k = 0$ for all $i \in I_0(x^*)$ and all $k$ sufficiently large. 
To this end, note that
tationarity of $(x^k,y^k)$ for \eqref{Penalpha} can be written as
\begin{equation} 
	0 = \nabla f(x^k) + g'(x^k)^T\lambda^k + \alpha_k \sum_{i \in I_0(x^*)} \gamma_i^k e_i + \alpha_k \sum_{i \in I_0(y^*)} \gamma_i^ke_i, \label{xkStat} 
\end{equation}
where $ \gamma_i^k \in y_i^k\partial(|x_i^k|)$ and $\lambda^k \in N_X^{\lim}(g(x^k))$. By part (i), we immediately have $\gamma_i^k = 0$ for $i \in I_0(y^*)$ and $k$ sufficiently large. Assume now there is $i \in I_0(x^*)$ such that $|x^k_i| \searrow 0$. Then, in particular, $ \gamma_i^k \in \{-y_i^k, y_i^k\}$, where $y_i^k \to y_i^* \neq 0$. In particular, 
this implies that the sequence $ \{ \| (\lambda^k, \alpha_k \gamma^k )
\| \} $ converges to $ \infty $ since $ \alpha_k \to \infty $.
Without loss of generality, we may assume that the corresponding
normalized sequence converges, say
$$
   ( \bar{\lambda}, \ \bar{\gamma}) = \lim_{k \to \infty} \frac{1}{\lVert (\lambda^k, \ \alpha_k \gamma^k)\rVert} (\lambda^k, \ \alpha_k\gamma^k).
$$
Using the cone property of $ N_X^{\lim} $ as well as the robustness
of the limiting normal cone, see \cite[Prop.\ 1.3]{Mor-2018}, we have $\bar{\lambda} \in N_X^{\lim}(g(x^*))$. Dividing (\ref{xkStat}) with $\lVert(\lambda^k, \ \alpha_k\gamma^k)\rVert$ and pushing to the limit $k \to \infty$, we therefore get
$$ 
	0 = g'(x^*)^T \bar{\lambda} + P^T\bar{\gamma}, \quad (\bar{\lambda}, \ \bar{\gamma}) \neq 0,
$$
a contradiction to the assumed SP-GMFCQ. Hence, we have $ x_i^k = 0 $
for all $ i \in I_0(x^*) $ and all sufficiently large $ k $.

As a consequence of parts (i) and (ii), it follows that, eventually,
$I_0(x^k) = I_0(x^*)$ and $I_0(y^k) = I_0(y^*)$ holds so that, in fact, $(x^k,y^k)$ is strictly complementary and, in particular, feasible for \eqref{SPOref}. 

It remains to see that $y^k$ takes the values $y_i^k = s_i^\rho \ (= y_i^*)$ for $i \notin I_0(y^*)$. This is true since, for all $i \notin I_0(y^*) = I_0(y^k)$, we obtain from \eqref{PenStat} that 
$$ 
   0 = \nabla p_i^\rho(y_i^k)
$$
holds, so that $ y_i^k $ is a minimum of the convex function
$p_i^{\rho} $. However, by assumption, $ p_i^{\rho} $ attains its unique
minimum at $ s_i^{\rho} $, so that $y_i^k = s_i^\rho$ follows
for all sufficiently large $ k $. 
This finishes the proof.
\end{proof}

The final result of this section is not directly connected to the previous exactness results. In a numerical setting, increasing the penalty parameter $\alpha$ should, regardless of exactness (or the satisfaction of
suitable constraint qualifications), minimize the residual $|x|^Ty$. 
This is, in fact, the case even if we are only able to find approximate stationary points.

\begin{thm}
Let $\varepsilon > 0$ such that there exists a $t^* > 0$ for which $\nabla p^\rho_i(t^*) > \varepsilon$. Let $\{(x^k,y^k)\}$, $\{z^k\}$ be sequences satisfying
$$
	g(x^k) + z^k \in X,  \quad z^k \in B_{\varepsilon}(0), \quad y_i^k ,
	\ge 0
$$
and assume that there are sequences $\{\lambda^k\} \subset N_X^{\lim}(g(x^k) + z^k)$ and $\gamma^k \in N_{\ge 0}(y^k)$ such that
$$
	\begin{pmatrix}
    -\nabla f(x^k) - \alpha_k y^k\circ \partial(|x^k|) - g(x^k)^T\lambda^k  \\
    -\nabla p^\rho(y^k) - \alpha_k|x^k| + \sum_{i\in I_0(y^k)}\gamma_i^ke_i 
	\end{pmatrix} \in B_{\varepsilon}(0)^2.
$$
If $\alpha_k \to \infty$, then necessarily $|x^k|\circ y^k \to 0$.
\end{thm}

\begin{proof}
It suffices to consider the assumption
$$
	-\nabla p^\rho(y^k) - \alpha_k|x^k| + 
	\sum_{i\in I_0(y^k)}\gamma_i^ke_i \in B_{\varepsilon}(0).
$$
This implies the existence of another bounded sequence $\{\delta^k\} \subset B_{\varepsilon}(0)$ such that
$$
	0 = \nabla p^\rho(y^k) + \alpha_k |x^k| - \sum_{i \in I_0(y^k)}\gamma_i^ke_i + \delta^k.
$$
Now multiply the above equation componentwise with $y^k$. We then obtain
for each $ i \not\in I_0(y^k) $ (otherwise there is nothing to show)
$$
	0 = y_i^k \nabla p_i^\rho(y_i^k) + \alpha_k y_i^k|x_i^k| +
	y_i^k\delta^k_i > y_i^k (\nabla p_i^\rho(y_i^k) - \varepsilon) +
	\alpha_ky_i^k|x_i^k|.
$$
Assume $y_i^k \to \infty$. Then $y_i^k > t^*$ and, by assumption and
convexity of $ p_i^{\rho} $, we then have  $\nabla p^\rho_i(y_i^k) - \varepsilon > 0$, which leads to a contradiction. Hence $y_i^k$ is bounded. Now assume there is a subsequence $K$ for which $y_i^k|x_i^k| > c$. Then $\alpha_ky_i^k|x_i^k| \to _K \infty$, which again is a contradiction.
Altogether, this shows that $ | x^k |^T y^k \to 0 $ for $ k \to \infty $.
\end{proof}

We only require the iterates $y^k$ from the exact penalty method to be nonnegative and $\nabla p_i^\rho(t)$ to become sufficiently large. Both is very easy to satisfy, since, on the on hand, we can simply project any iterate $y^k$ onto $\mathbb{R}_{\ge 0}^n$ and, on the other hand, choose $\nabla p^\rho$ in a way such that the derivative assumption from the 
previous result holds. Note that this assumption is satisfied for
all mappings from Example~\ref{Ex:pirho}.

\section{Solution Strategies for Penalized Problems} \label{section:solstrat}

The solution of the penalized subproblems requires a beneficial structure
of the constraints $g(x) \in X$. For example, for $X = \mathbb{R}^n_{\le0}$, one could simply augment the constraints by changing the objective function to
$$
	F(x) = f(x) + \mu \max\{0,g(x)\}^2,
$$
and incorporate an augmented Lagrangian method. On the other hand, for 
$g(x) = x$ and the set $X$ being nonempty, closed and convex, which admits a suitable projection, one could approach \eqref{Penalpha} with a projected gradient or proximal gradient method. We are mostly concerned with the latter scenario. To this end, we first consider an unconstrained version of \eqref{Penalpha} and present two unique ways of applying first-order methods to the penalty formulation. These methods have to be slightly adapted in a practical setting to accommodate for a possible projection, which was no issue in our numerical test examples.

\subsection{A Projected Spectral Gradient Method}

We rewrite problem \eqref{Penalpha} with the introduction of an auxiliary variable $s$ in the form
\begin{equation}
    \min_{x,s,y} f(x) + p^\rho(y) + \alpha s^Ty, \st y \ge 0, \ |x| \le s. \label{PenEpi}
\end{equation}
The problem \eqref{PenEpi} has a convex feasible set and is equivalent to \eqref{Penalpha} in the sense that both objective functions take the same exact value if, for $y_i \neq 0$, we have $|x_i| = s_i$, which is also necessary at a local minimum of \eqref{PenEpi}. However, $|x_i| < s_i$ does not change the target value of \eqref{PenEpi} whenever $y_i = 0$, which shows that \eqref{PenEpi}, in general, produces infinitely many local minima, so one may be interested in using the more rigorous constraint 
$|x| = s$. This feasible set, however, is no longer convex and computationally less preferable. We rather propose to simply modify a solution approach to \eqref{PenEpi} in such a way that all iterates
satsify the constraint $|x| = s$, which is not hard to achieve since simply overwriting $s$ to the value of $|x|$ leads, in any case, to a better target value of \eqref{PenEpi}.

As mentioned before, the set $\Omega := \{(x,s,y) \ : \ |x| \le s, \ y \ge 0\}$ is convex since, by reordering the entries of $(x,s,y)$, one may regard $\Omega$ as the Cartesian product
$$
	\Omega = \Pi_{i = 1}^n\text{epi}(|\cdot|) \times \RR^n_+,
$$
where $\text{epi}(|\cdot|)$ denotes the epigraph of the absolute value function:
$$
	\text{epi}(|\cdot|) := \{(x,s) \ : \ |x| \le s\}.
$$
The latter is convex since the absolute value function is convex. Then it is easy to compute the projection onto $\Omega$.

\begin{lem}
Let $\Omega = \{(x,s,y) \ : \ |x| \le s, \ y \ge 0\}$. Then the projection $(x,s,y) = P_{\Omega}(u,v,w)$ is given by
$$(x_i,s_i) = \begin{cases} (u_i, v_i), \quad &\text{if} \quad |u_i| \le v_i, \\
(0,0), \quad &\text{if} \quad |u_i| \le -v_i,\\
\frac{1}{2}(u_i + v_i, u_i + v_i), \quad &\text{if} \quad u_i > |v_i|, \\
\frac{1}{2}(u_i - v_i, v_i - u_i), \quad &\text{if} \quad u_i < -|v_i|,\end{cases}$$
and $y = \max\{w, 0\}$.
\end{lem}

The following algorithm is mainly the spectral gradient method (SPG) as presented in \cite{Birgin2014}, which we copy over for the most part, but overwrite the values of $s^k$, obtained in each iteration, by the preferable value of $|x^k|$. For ease of notation, we will refer to the target function of \eqref{PenEpi} by $F(x,s,y)$.

\begin{algo}[SPG] \label{SPG}
    Let $\beta \in (0,1), \ 0<\sigma_{\min}<\sigma_{\max}$, and $M$ be a positive integer. Let $(x^0,s^0,y^0) \in \Omega$ be an arbitrary initial point. For $k = 0, 1, 2, ...$
    \begin{enumerate}
        \item Compute $\sigma_k^{SPG} \in [\sigma_{\min}, \sigma_{\max}]$ by
        $$\sigma_k^{SPG} = \begin{cases}
            1 \quad &\text{if} \quad k = 0,\\
            \max\left\{ \sigma_{\min}, \min\left\{ \frac{(v^k)^Tw^k}{(v^k)^Tv^k}, \sigma_{\max}\right\}\right\} \quad &\text{otherwise},
        \end{cases}$$ where $v^k = (x^k,s^k,y^k) - (x^{k-1}, s^{k-1}, y^{k-1})$ and $w^k = \nabla F(x^k,s^k,y^k) - \nabla F(x^{k-1},s^{k-1},y^{k-1})$, and set 
        $$d^k := P_{\Omega}\left((x^k,s^k,y^k) - \frac{1}{\sigma_k^{SPG}}\nabla F(x^k, s^k, y^k)\right) - x^k $$
        \item Set $t \leftarrow 1$ and $F^{ref}_k = \max\{F(x^{k-j+1}, s^{k-j+1},y^{k-j+1}) \ : \ 1 \le j \le \min\{k+1,M\}\}.$ If
        \begin{equation} \label{StepSize}
            F((x^k,s^k,y^k) + t d^k) \le F_k^{ref} + t \beta \nabla F(x^k,s^k,y^k)^Td^k,
        \end{equation}
        set $t_k = t$ and set
        \begin{align}
            \begin{split}
                x^{k+1} &= x^k + t_k (d^k)_{1:n},\\
                s^{k+1} &= |x^{k+1}|,\\
                y^{k+1} &= y^k + t_k (d^k)_{2n+1 : 3n},
            \end{split}
        \end{align}
        and finish the iteration. Otherwise, choose $t_{\text{new}} \in [0.1t, 0.5t]$, set $t \leftarrow t_{\text{new}}$ and repeat the test (\ref{StepSize}).
    \end{enumerate}
\end{algo}

Notice that the update of $s^k$ by $|x^k|$ is in line with \cite{Birgin2014} as it is easy to see that in any case for a feasible tuple $(x,s,y)$ it holds $F(x,|x|,y) \le F(x,s,y)$. By continuity of the absolute value function, any accumulation point $(x^*,s^*,y^*)$ obtained from \eqref{SPG}, in particular, satisfies $s^* = |x^*|$. Notice that \eqref{PenEpi} is, in fact, a smooth NLP, as one may simply write
$$
	(x,s) \in \text{epi}(|\cdot|) \quad \Leftrightarrow \quad x - s\le 0\  \wedge \ -x-s\le0. 
$$
This implies that any other suitable smooth NLP solver can be used and,
under additional constraints of the type $h(x) = 0, \ g(x) \le 0$, Lagrangian-type methods are applicable or Lagrangian stationarity can be used to derive the Euclidean projection onto the feasible set.

\subsection{A Proximal Point Method}

We may also regard the problem \eqref{Penalpha} as a composite optimization problem, where $f_1(x,y):= f(x) + p^\rho(y)$ is the smooth part and $f_2(x,y):= \alpha |x|^Ty $ the nonsmooth function for which we have to compute the prox-operator
$$
	\text{prox}^{SP}_\gamma(u,v) := \text{arg}\min_{\substack{x,y \\ y\ge 0}} f_2(x,y) + \frac{1}{2\gamma} \norm{(x,y) - (u,v)}_2^2.  
$$
By separability of $f_2$, we have
$$
	f_2(x,y) + \frac{1}{2\gamma} \norm{(x,y)-(u,v)}_2^2 = \sum_{i = 1}^n \alpha|x_i|y_i + \frac{1}{2\gamma} \left( (x_i - u_i)^2 + (y_i - v_i)^2\right).
$$
At a solution $(x^*,y^*)$, we can always infer $\text{sign}(x_i^*) = \text{sign}(u_i)$ so that we only have to solve the optimization problem
\begin{equation} \label{ProxProb}
    \min_{x,y} \alpha xy + \frac{1}{2\gamma}\left((x - u)^2 + (y - v)^2 \right) \st x\ge0, \ y \ge0,
\end{equation}
where $u \ge 0$.

\begin{lem}
Let $u \ge 0$ and $v \in \RR$. The solution to problem (\ref{ProxProb}) is given by
$$
	(x,y) = \begin{cases}
    (u,0), \quad &\text{if} \quad v<0,\\
    \frac{1}{1 - \gamma^2 \alpha^2} (u - \gamma\alpha v, v - \gamma \alpha u), \quad &\text{if} \quad \frac{1}{\gamma} > \alpha \ \text{and}\ u \in [\gamma \alpha v, \frac{v}{\gamma\alpha}], \\
    (0,v) \frac{\text{sign}(v - u) + 1}{2} + (u,0) \frac{1 - \text{sign}(v - u)}{2}, \quad &\text{otherwise}.
\end{cases}
$$
\end{lem}

\begin{proof}
Assume $v < 0$. For a fixed $x$, the best possible value of $y$ for \eqref{ProxProb} is clearly $ y = 0$. In that case, however, $x$ 
is necessarily given by $x = u$.

Assume for the remainder $v \ge 0$. First notice that \eqref{ProxProb} is coercive on the feasible set which implies that a solution always exist. 
We discern the following cases.
\begin{enumerate}
    \item Let $1 / \gamma > \alpha$, then the target function of (\ref{ProxProb}) is uniformly convex. Furthermore, for $u \in [\gamma \alpha v, \frac{v}{\gamma\alpha}]$, we have
\begin{equation*}
    1 - \gamma^2 \alpha^2 > 0, \quad u - \gamma\alpha v\ge 0, \quad v - \gamma \alpha u \ge 0,
\end{equation*}
and the point 
$$
	(x,y) = \frac{1}{1 - \gamma^2 \alpha^2} (u - \gamma\alpha v, v - \gamma \alpha u)
$$
is both a stationary and feasible point of the objective function and hence the solution.
\item Let $1/\gamma > \alpha$ and $u \notin [\gamma \alpha v, \frac{v}{\gamma \alpha}]$. Then the target function of \eqref{ProxProb} has only one (unconstrained) stationary point, which is not feasible. 
Therefore, the solution must lie on the boundary, so that either $(x,y) = (u,0)$ if $v < u$ or $(x,y) = (0,v)$ if $u < v$.
\item Let $1 / \gamma = \alpha$. Then the target function has infinitely many stationary points if and only if $u = v$ with condition $(x + y) = u$, or there are no (unconstrained) stationary points. In any case, the minimal value is attained at a boundary point $(x,y) = (u,0)$ if $v < u$ or $(x,y) = (0,v)$ if $u \le v$.

\item         Finally, let $1 / \gamma < \alpha$. The target function has only one stationary point, which is necessarily a saddle point. The only solutions are found on the boundary so that again $(x,y) = (u,0)$ if $v < u$ or $(x,y) = (0,v)$ if $u \le v$.
\end{enumerate}
Altogether the desired representation follows.
\end{proof}

The proximal gradient method is then implemented with a nonmonotone stepsize as for instance seen in \cite{DeMarchi2023}.

Note that we have only stated the projection and, respectively, the proximal point operator in an unconstrained case, to show that either method is very well realizable. Depending on the optimization problem, of course, modifications are necessary, or additional constraints have to, for instance, be augmented and solved via an augmented Lagrangian approach.

\section{Numerical Results}\label{section:numerics}

\subsection{Sparse Portfolio Optimization}
\subsubsection{Problem Formulation}

To measure the quality of a portfolio of investments, one possible formulation of the Markov model \cite{Markowitz1952} takes an estimated risk, given by some covariance matrix $Q$, and compares it to a possible return, given by some vector $\mu$. The investor has a certain amount of capital to work with, and we assume, additionally, that short sales are possible. We therefore consider the problem
\begin{equation}
    \min_x \ \frac{1}{2}x^TQx - \beta\mu^Tx + \rho\norm{ x }_0 \st e^Tx = 1. \label{SpPort}
\end{equation}
Since the investor might be dissuaded from portfolio strategies with a wide spread, a sparsity-inducing penalty term is in place. Using the $\ell_1$-norm, at any feasible point, we have
$$
	\norm{x}_1 = \sum_{x_i \ge 0} x_i + \sum_{x_i < 0} |x_i|, \quad \sum_{x_i\ge 0}x_i -\sum_{x_i <0}|x_i| = 1 .
$$
Hence, a spread without shorting always leads to $\norm{x}_1 = 1$, with possibly no sparsity in $x$.
For a spread with shorting, on the other hand, we necessarily have $\norm{x}_1 > 1$. It follows that the $\ell_1$-regularization is a bad choice, since it favors spreads without shorting and possibly with no sparsity.

It is well known that the above problem \eqref{SpPort} can be 
represented as a quadratic mixed integer problem of the form
$$
	\min_{\substack{x\in\mathbb{R}^n, \\ y\in\{0,1\}^n}} \ \frac{1}{2}x^TQx - \beta\mu^Tx + \rho\norm{y}_1 \st e^Tx = 1, \ y_il_i\le x_i \le y_iu_i, \ i = 1,...,n,
$$ 
where $l$ and $u$ denote some (sufficiently large) lower and upper bounds
satisfying $l \le 0 \le u$. Specialized combinatorial solver can be applied to solve the mixed integer formulation, for instance, Gurobi.

To solve (\ref{SpPort}) with our exact penalty approach, we focus on the spectral gradient method. The reason for this is that we require a projection onto the set
\begin{equation}
	X = \{(x,s) : \ e^Tx = 1, \ x - s \le0, \ -x-s\le 0\}. \label{portfolioset}
\end{equation}
Assume we want to project the vector $(a,b)$. Necessary and sufficient for the solution $(x^*,s^*)$ is the existence of multipliers $\lambda^+, \ \lambda^- \in \mathbb{R}^n_{\ge 0}$ and $\mu \in \RR$ that fulfill
\begin{align}\begin{pmatrix}
    x^* - a + \mu e + \sum \lambda_i^+e_i - \sum \lambda_i^-e_i \\ s^* - b -\sum\lambda_i^+ e_i - \sum \lambda_i^-e_i
\end{pmatrix} &= 0,\label{eqn:KKT} \\e^Tx^*&=1 \label{eqn:eqconst1},\\
\lambda^+ \ge 0, \quad s^*-x^* \ge 0, \quad (s^*-x^*)^T\lambda^+ &= 0, \label{eqn:comp1}\\
\lambda^- \ge0, \quad s^* +x^* \ge 0, \quad (s^* +x^*)^T\lambda^- &= 0. \label{eqn:comp2}
\end{align}
A simple calculation shows that we require
$$
	s^* - x^* - 2\lambda^+ = b - a + \mu e,\quad \text{and} \quad s^* + x^* - 2\lambda^- = a + b - \mu e.
$$
From the complementarity in \eqref{eqn:comp1} and \eqref{eqn:comp2}, we have
\begin{align*}
    s^* - x^* &= \max\{0, b - a + \mu e\}, & s^* + x^* &= \max\{0, a + b - \mu e\}, \\
    2\lambda^+ &= \max\{0,a-b-\mu e\}, & 2\lambda^- &= \max\{0, \mu e - a - b\}.
\end{align*}
By \eqref{eqn:eqconst1}, it holds
$$
	e^T(s^* + x^*) - e^T(s^* - x^*) = 2
$$
or, equivalently, we need $\mu$ as a root of 
\begin{equation}
    t(u):=\sum_i \max\{0,a_i +b_i - u\} - \max\{0,b_i - a_i + u\} - 2. \label{eqn:auxroot}
\end{equation}
It is easy to see that $t$ is continuous and monotonically decreasing. A root $\mu$ of $t$ must exist, since clearly there is a solution to \eqref{eqn:KKT}-\eqref{eqn:comp2}. Set
$$
	u_- := \max\{\max_i\{a_i + b_i\}, \max_i\{a_i - b_i\}\}, \quad u_+:=\min\{\min_i\{a_i +b_i\}, \min_i\{a_i -b_i\}\} - \frac 2 n.
$$
Plugging the two values into $t$, one has $t(u_-) \le 0$ and $t(u_+)\ge 0$ and, therefore, $u_-$ and $u_+$ can be used as the start of a bisection method to obtain $\mu$. We have the following Lemma.

\begin{lem}
    Consider the set $X$ in \eqref{portfolioset} and let $(x^*,s^*) = P_X(a,b)$ denote the euclidean projection of $(a,b)$ onto $X$. Then
    \begin{align*}x^* &= \frac 1 2 \left(\max\{0,a + b - \mu e\} - \max\{0, b - a + \mu e\}\right),\\
    s^* &= \frac 1 2 \left(\max\{0, b - a + \mu e\} + \max\{0,a + b - \mu e\}\right),\end{align*}
    where $\mu$ is a root of $t$ as given in \eqref{eqn:auxroot}.
\end{lem}

\subsubsection{Numerical Tests}
For the numerical test, we use an instance of 90 generated covariance matrices $Q$ and return vectors $\mu$ from dimensions $n = 200$, $n = 300$ and $n = 400$ provided by Frangioni and Gentile\footnote{\url{http://groups.di.unipi.it/optimize/Data/MV.html}}. We used Gurobi\footnote{\url{https://www.gurobi.com/}} to determine a best possible value of each test problem. Gurobi was set to run for at least 60 seconds in all of the 90 cases and we compared the results to our exact penalty method. The computation were carried out via Python, where we chose $\alpha_0 = \rho = \beta = 1$.  The spectral gradient was run for $1000$ iterations or until a stationarity to the tolerance of $10^{-4}$ was reached. In the outer loop of the exact penalty method, we set $\alpha_{k+1} = 2 \alpha_k$ until a complementarity $\lVert s^k \circ y^k\rVert_{\infty} <10^{-3}$ was reached. On average, the required computation time for a full run of the exact penalty method for a single instance was
$$
	t_{200} = 3.94s, \quad t_{300} = 5.69s, \quad t_{400} = 6.66s,
$$
for dimensions $200$ to $400$, respectively (Figure \ref{PortfolioBench}).

\begin{figure}[H]
\centering
    \begin{subfigure}{\linewidth}
        \centering
        \includegraphics[width =\linewidth]{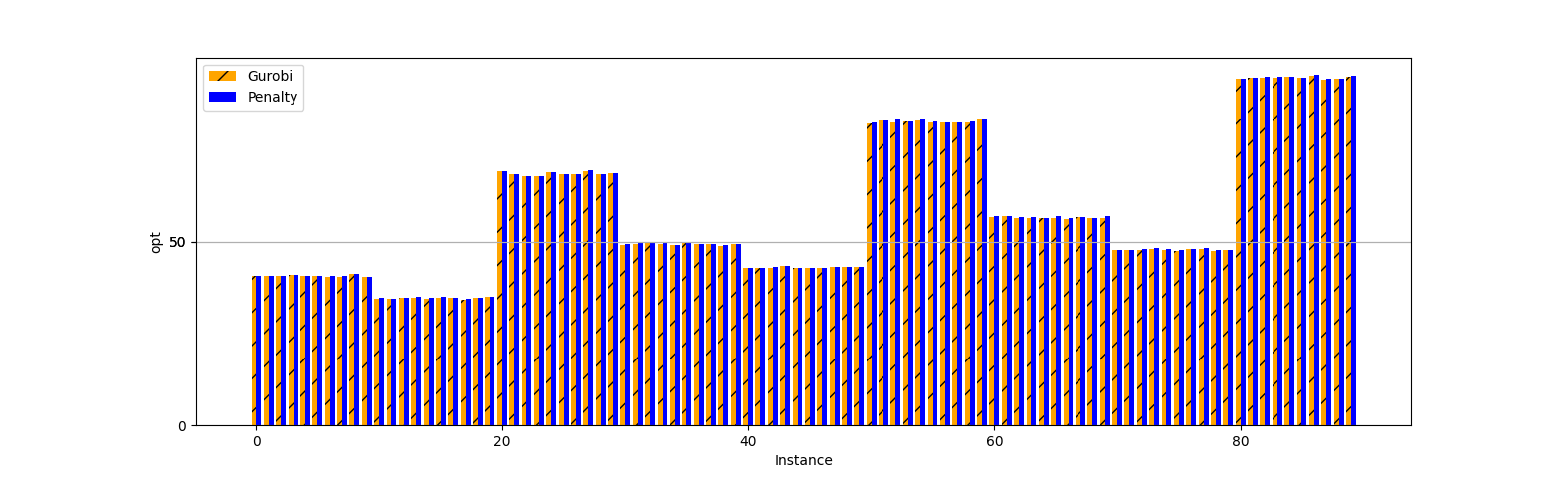}
        \subcaption{Optimal values of $f(x) + \lVert x\rVert_0$ reached by Gurobi and the penalty approach}
        \label{fig:GurobiVsPenaltyHuber}
    \end{subfigure}
    \begin{subfigure}{\linewidth}
        \centering
        \includegraphics[width = \linewidth]{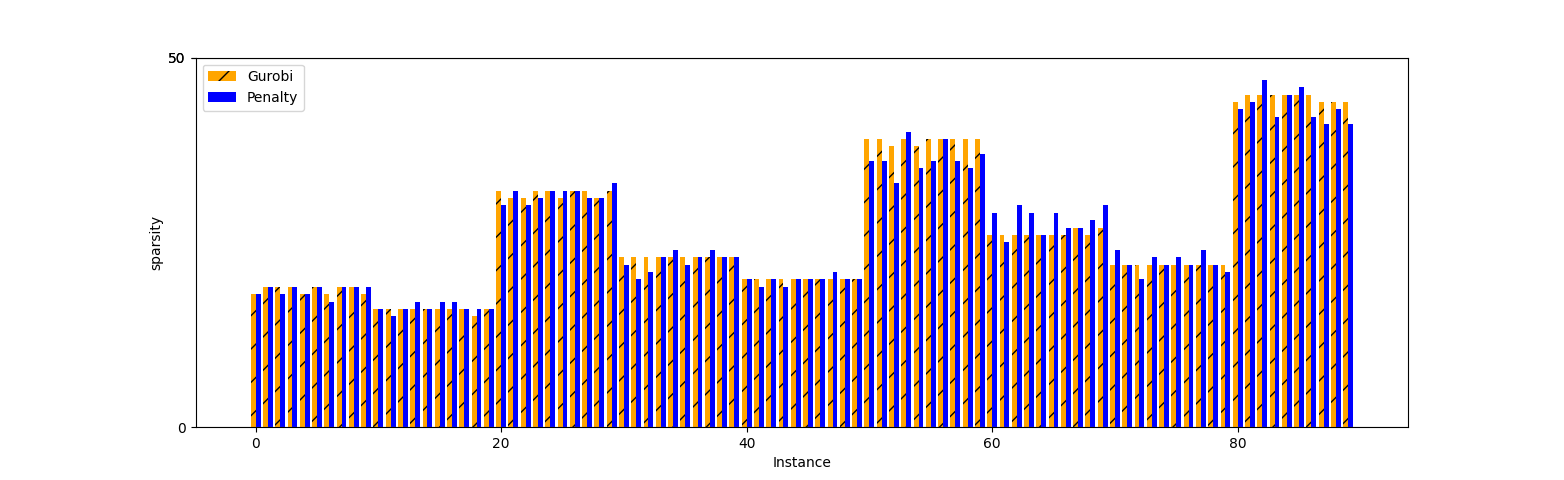}
        \subcaption{Sparsity reached by Gurobi and the penalty approach}
        \label{fig:GurobiVsPenalty(b)}
    \end{subfigure}
    \caption{Overview of the sparse portfolio tests.}
    \label{PortfolioBench}
\end{figure}

\subsection{Sparse Dictionary Learning}

\subsubsection{Problem Formulation}

We consider the matrix-valued optimization problem
$$
	\min_{C,D} \frac{1}{2} \left\lVert {D^TC - Z} \right\rVert_F^2 + \rho \norm {C}_0\st \lVert d_i^T\rVert_2 \le 1,
$$
where $D \in \mathbb{R}^{l\times n}$ is considered a dictionary with rows $d_i^T$ from which a couple of entries are drawn to create matrix $Z \in \mathbb{R}^{n\times m}$. This process is represented by multiplying $D^T$ with a suitable sparse matrix $C$, which gives rise to the formulation with the sparsity inducing $\ell_0$-norm regularization term. Note that, in this context, we use
$$
	\norm {C}_0 = \norm{ \text{vec}(C) }_0.
$$
For this class of problems, the feasible set does not give any restrictions for $C$, the variable we impose sparsity on, therefore projection onto the epigraph of the componentwise absolute value function as well as the proximal operator $\text{prox}^{SP}_\gamma$ can be computed in a 
straightforward manner. Since replacing the $\ell_0$-norm by the $\ell_1$-norm is a popular approach that turns the objective function into a biconvex optimization problem, we compare in the following the results achieved by a hard-thresholding approach, using the proximal operator of the $\ell_0$-norm to a soft-thresholding approach, using the proximal operator of the $\ell_1$-norm, and finally the exact penalty with spectral gradient and the proximal point method.

\subsubsection{Numerical Tests}

We copy the setup from \cite{DeMarchi2023} and create $100$ test problems, with $n = 100$, $l = 200$ and $m = 300$. The entries of $C^0$ and $D^0$ are drawn from a standard normal distribution $\mathcal{N}(0,1)$, projected onto the feasible set. Again, for simplicity, we set $\rho = 1$. In case of the exact penalty method, the inner solver was set to run for $10^4$ iterations or a until a stationarity of tolerance $10^{-5}$ was reached. We also set $\alpha_0= 1$ and $\alpha_{k+1} =1.5\, \alpha_k$ until 
$$
	\langle |C^k|, Y^k\rangle_F := \text{trace} ((C^k)^T Y^k) \le 10^{-3}.
$$
In case of the hard and soft thresholding method, we set the maximum number of iterations to $10^5$ and the tolerance of stationarity to $10^{-6}$. We draw the performance profiles regarding the optimal function values reached and the required computational time.
\begin{figure}[ht]
\centering
    \begin{subfigure}{.4\linewidth}
        \includegraphics[width = \linewidth]{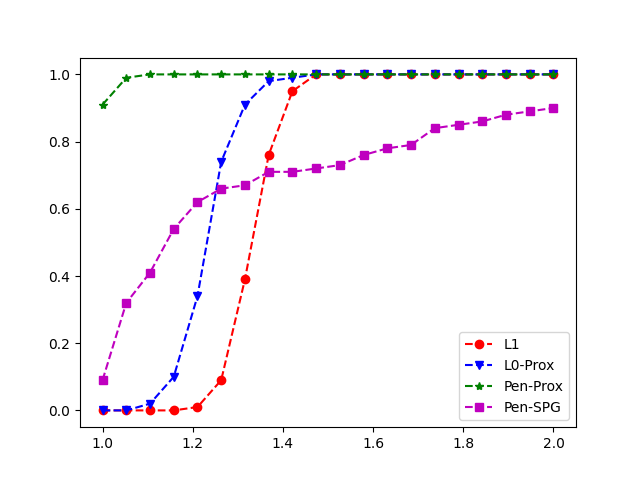}
        \subcaption{Performance profile on the optimal target value.}
        \label{fig:DictVal}
    \end{subfigure}
    \begin{subfigure}{.4\linewidth}
        \includegraphics[width = \linewidth]{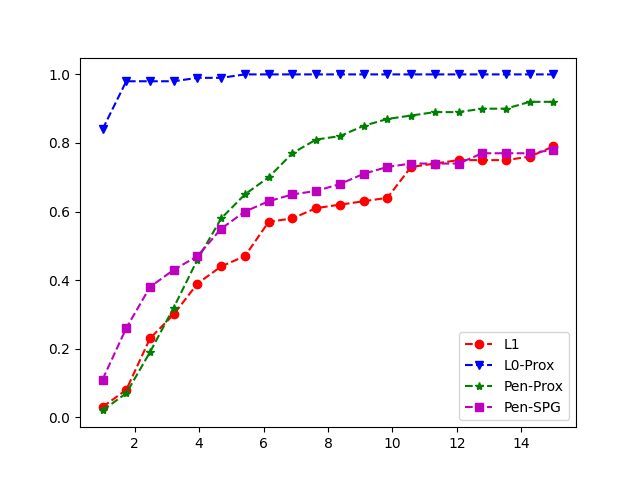}
        \subcaption{Performance profile on the computation time.}
        \label{fig:DictTime}
    \end{subfigure}
    \caption{Overview of the sparse portfolio tests.}
\end{figure}
As it turns out, the exact penalty method with the inner proximal point solver (Pen-Prox) has a chance of approximately $90\%$ of achieving the best possible value, while the spectral gradient (Pen-SPG) was able to find the best possible value in $10\%$ of the cases (Figure \ref{fig:DictVal}). Both thresholding methods were never able to find the best possible solution. The spectral gradient manages to stay in the range of the optimal solution for approximately $60\%$ of all cases and is then outperformed by the solver $\ell_0$ and $\ell_1$.

Considering computational time, the $\ell_0$ solver wins the race, which was to be expected. Surprisingly enough, the penalty approaches stay well within the performance of the soft-thresholding operator, with the Pen-Prox overtaking and landing at $80\%$ chance to require roughly ten times the computational time of the hard-thresholding solver (Figure \ref{fig:DictTime}).

\subsection{Sparse Untargeted Adversarial Attacks on Neural Networks}

\subsubsection{Problem Formulation}
Let $M := \{(x^d,y^d) \in \mathcal{X} \times \mathcal{Y}, \ d = 1,...,S\}$ denote a set of data samples, with sample space $\mathcal{X} \subset \mathbb{R}^n$ and label space $\mathcal{Y} \subset \mathbb{R}^m$. A neural network is a function $\phi \ : \ \mathbb{R}^n \to \mathbb{R}^m$, which aims to satisfy $$P\left(\phi(x^d)\right) = y^d, \quad \forall d =1,...,S,$$
where $\phi$ is a composition of several functions
$$\phi = f_l \circ f_{l-1} \circ \cdots \circ f_1,$$
each corresponding to a layer $k = 1,...,l$, comprised of a weight matrix $W^k$, a bias $b^k$ and some activation function $\sigma_k \, : \, \mathbb{R}^{n_k} \to \mathbb{R}^{m_k}$ so that essentially $f_{k}(x) = \sigma_{k}(W^k f_{k-1}(x) + b^k)$ and $f_0 = \text{id}$. The function $P(\cdot)$ processes the vector of logits (the raw output of network $\phi$) in a suitable way. Training of a neural network is carried out by the minimization of some loss function $L$ across the data samples in $M$, in the sense that one has to solve
$$\min_{\{(W^k,b^k), \ k = 1,...,l\}} \frac{1}{S}\sum_{d = 1}^S L\left(P(\phi(x^d)),y^d\right).$$
Assume that $y^d$ is a probability mass function, which belongs to some discrete distribution, that is $\mathcal{Y} =\{y \in \mathbb{R}^m \ : \ y \ge 0, \ \lVert y \rVert _1 = 1\}$. One chooses $P$ as a function to carry the logits into the probability space $P(\phi(x^d)) = p^d$, for instance, by the softmax
$$p^d = \text{softmax}(\phi(x^d)) :=\frac{z^d}{\norm{ z^d }_1}, \quad \text{with} \quad z^d_i = \exp\left(\phi(x^d)_i\right),$$
to then minimize a statistical distance between $y^d$ and $p^d$, where now $p_i^d$ measures the probability of $x^d$ belonging to class $i \in \{1,...,m\}$.  For this purpose, one usually considers the Kullback-Leibler divergence
$$KL(y^d,p^d) := \sum_{i = 1}^m y_i^d \log\left(\frac{y_i^d}{p_i^d}\right),$$
and, since $y^d$ is fixed, we may simply choose $L(P(\phi(x^d)),y^d) := - \sum_{i = 1}^my_i^d \log(p_i^d),$ which is commonly called the 
\emph{cross entropy loss}. 
We now assume $y^d$ to be one-hot encoded, which means that $y^d= e_{i_d}$, i.e. $y^d$ is the $i_d$-th unit vector of $\mathbb{R}^m$. The model $\phi$ should then satisfy
$$\underset{i = 1,...m}{\text{argmax}}\  P(\phi(x^d)) = i_d.$$

In the sparse adversarial attack setting, the model $\phi$ has already been trained and we try to find a suitable deviation $\delta$, where $\lVert \delta \rVert_0$ is small, to throw off the previous classification of an input $x^d$, that is
$$\underset{i = 1,...m}{\text{argmax}}\  P(\phi(x^d + \delta)) \neq \underset{i = 1,...,m}{\text{argmax}}\ P(\phi(x^d)). $$
Since the latter cannot be chosen as a constraint, we instead want to minimize the auxiliary function
$$
	F(\delta) = \log(p_{i_d}^d(\delta)) - \sum_{\substack{j = 1 \\ j \neq i_d}}^m \log(p_j^d(\delta)), \quad \text{with} \quad p^d(\delta) = P(\phi(x^d + \delta)).
$$
This reflects the idea that, by minimizing the cross entropy loss, $\log(p_{i_d}^d)$ has to have been large, since $i_d$ was the predicted label, whereas the logarithm with respect to the other classes $\{1,...,n\} \setminus\{i_d\}$, must have been small. This, of course, is simply a heuristic approach, as it is, for instance, unclear whether an appropriate scaling between the two summands in $F$ should be in place. We now try to find solutions $\delta^\rho$ of
$$
	\min_{\delta} \frac{1}{\rho} F(\delta) + \norm{ \delta }_0, \st x+\delta \in \mathcal{X}
$$
for decreasing values of $\rho$ until for a solution $\delta^\rho$, 
the input $x^d + \delta^\rho$ is classified differently from $x^d$. This approach was presented in \cite{Carlini2017}, albeit with a different selection of penalty functions $F$ and, additionally, the authors of the aforementioned work focused on targeted network attacks, where a certain class was given, $x^d$ should be recognized as. This might in some cases lead to better results, since a model tasked with discerning, for instance, numbers might confuse a one with a seven easier than a one with an eight.

\subsubsection{Numerical Test}

We execute the adversarial attacks on the well known MNIST data set, which consist of $60000$ hand-drawn black and white images with $28 \times 28$ pixels, that is $\mathcal{X} = [0,1]^{28 \times 28}$,
depicting numbers from $0$ to $9$. The set of labels can therefore be represented as $\mathcal{Y} = \{0,1\}^{10}$.

We set up a neural network of the forward structure as specified in 
Table~\ref{tab:MNIST}. We chose a dropout of $0.2$ and activated the convolution and fully connected layers with the sigmoid function each.

\begin{center}
    \begin{table}[H]
        \centering
        \begin{tabular}{|c|}\hline
             Convolution 2D layer with 32 filters and kernel size 3 \\ \hline
             Convolution 2D layer with 32 filters and kernel size 3\\ \hline
             Average Pooling 2D layer of size $2 \times 2$\\ \hline
             Convolution 2D layer with 64 filters and kernel size 3 \\ \hline
             Convolution 2D layer with 64 filters and kernel size 3 \\ \hline
             Average Pooling 2D layer of size $2 \times 2$\\ \hline
             Fully connected layer of size $200$\\ \hline
             Fully connected layer of size $200$ \\ \hline
             Dense output layer of size $10$\\ \hline
        \end{tabular}
        \caption{Layers of the MNIST Model}
        \label{tab:MNIST}
    \end{table}
\end{center}
The training was carried out via Tensorflow\footnote{\url{https://www.tensorflow.org/}}, which is freely available as a python package, and computation was done on a Nvidia RTX 3070 GPU, where we attained an accuracy of $99.92 \%$ and a loss of $0.0022$ across the training set, and an accuracy of $99.22 \%$ and a loss of $0.0352$ on the test set.

We deployed an adversarial attack on the first $100$ data samples of the training set. To this end, we chose $\rho^0 = 10$ and $\rho_{k+1} = 0.9 \cdot \rho_k$ until a solution $\delta^{\rho_k}$ of
$$
	\min_\delta \ \frac{1}{\rho_k} F(\delta) + \norm{ \delta }_0, \st 0\le x^d + \delta\le1 
$$
would satisfy
$$
	\underset{i = 1,...m}{\text{argmax}}\  P(\phi(x^d + \delta)) \neq \underset{i = 1,...,m}{\text{argmax}}\ P(\phi(x^d)). 
$$
The problems were solved with, on the one hand, a straightforward application of the $\ell_0$ proximal gradient method (L0-Prox) and, on the other hand, the exact penalty method with the spectral gradient inner solver (Pen-SPG) and proximal inner solver (Pen-Prox). To compute $\nabla F$, we rely on the automatic differentiation functionality of Tensorflow to obtain the gradient of $\phi$ with respect to the input $x^d + \delta$. The exact penalty parameter was initialized to $\alpha^0 = 1.0$ and increased in increments of $10$ until complementarity was reached to a tolerance of $10^{-3}$. The inner solvers were set to run for $10^3$ iterations until either stationarity of tolerance $10^{-4}$ was reached or the algorithm did no longer measure any progress in the attained function values over $10$ iterations.

Notice that we had to slightly adapt projections and prox-operators to incorporate the projection of $\delta$ onto $\delta_{ij} \in [-x_{ij}^d, 1-x_{ij}^d]$. It turns out that Pen-SPG behave significantly better than Pen-Prox, while also beating the hard-threshold approach L0-Prox in most of the cases. Pen-SPG generated adversarial images with on average $\norm{ \delta }_0 \approx11$, whereas L0-Prox generated adversarial images with $\norm{ \delta }_0 \approx 17$ and Pen-Prox coming in last with $\norm{ \delta }_0 \approx 29$. The performance profile on the attained sparsity of $\delta$ and the required computational time is detailed in Figure \ref{fig:AdvVal} and Figure \ref{fig:AdvTime}. Not surprisingly, the hard thresholding operator is, again, significantly ahead in terms of computational time.
    
\begin{figure}[ht]
\centering
    \begin{subfigure}{.4\linewidth}
        \includegraphics[width = \linewidth]{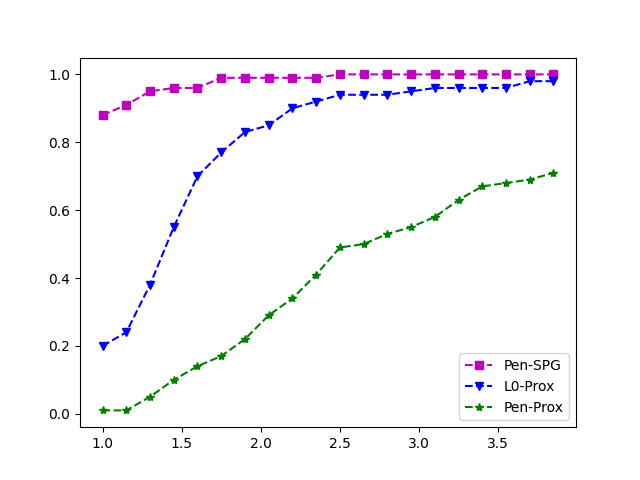}
        \subcaption{Performance profile on the optimal target value.}
        \label{fig:AdvVal}
    \end{subfigure}
    \begin{subfigure}{.4\linewidth}
        \includegraphics[width = \linewidth]{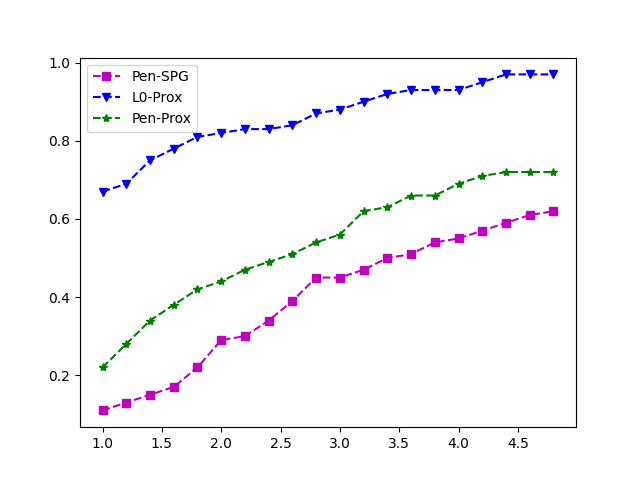}
        \subcaption{Performance profile on the computation time.}
        \label{fig:AdvTime}
    \end{subfigure}
    \caption{Overview of the adversarial attack tests.}
\end{figure}

\section{Final Remarks}\label{Sec:FinalRemarks}

This paper generalizes the authors' previous work \cite{KanzowWeiß2023}
and presents an exact penalty technique for sparse optimization
problems where the variables are not necessarily sign-constrained.
The theoretical results provide a very strong relation between 
the exact penalty approach and the given sparse optimization problem,
even in terms of stationary points, and the numerical results look
very promising.

In contrast to \cite{KanzowWeiß2023}, however, the newly introduced penalty
function is nonsmooth like most exact penalty functions for standard
nonlinear programs. On the other had, it is known that there exist
differentiable exact penalty techniques, see, e.g., the survey
article \cite{DiPillo-1994}. These differentiable exact penalty
methods have a very strong theoretical background, but their
practical application is limited due to the fact that their
evaluation is, in general, very expensive. In our particular
reformulation of the sparsity term, however, the resulting
constraints have a very simple structure, in which case the
differentiable exact penalty functions might be implemented in a 
much more efficient way than for general constraints. We therefore
leave it as part of our future research to investigate the 
application of differentiable exact penalty functions to our particular
reformulation of sparse optimization problems.

\section*{Statements}

\subsection*{Competing Interests}

There are no competing interests to report.

\subsection*{Data Availability}

We summarize the data availability for each of the considered test problems in the list below.
\begin{enumerate}
    \item The mixed-integer portfolio optimization problems by Frangioni and Gentile are publicly available (c.f. footnote 1). To our knowledge, the problems are randomly generated. The specific instance for this manuscript lies with the authors and is available on request.
    \item The sparse dictionary learning instances have been generated procedurally and randomly from a fixed random seed with the NumPy\footnote{\url{https://numpy.org/}} package available for Python. 
    \item The MNIST data set for the considered adversarial attack tests is publicly available and, in our case, was downloaded directly from TensorFlow. The CNN used for the classification task was trained as stated and saved as keras file. It is available on request.
\end{enumerate}

\bibliographystyle{habbrv}
\bibliography{references_penalty}
\end{document}